\documentclass[12pt,a4paper]{article}
\usepackage{amssymb,amsmath,amsthm}
\usepackage{latexsym}

\newtheorem{theorem}{Theorem}
\newtheorem{lemma}{Lemma}
\newtheorem{proposition}{Proposition}
\newtheorem{corollary}{Corollary}
\newtheorem{remark}{Remark}
\numberwithin{equation}{section}
\numberwithin{theorem}{section}
\numberwithin{lemma}{section}
\numberwithin{proposition}{section}
\numberwithin{corollary}{section}
\numberwithin{remark}{section}

\begin{document}
\title{Large time behavior of a generalized Oseen evolution operator,
with applications to the Navier-Stokes flow past a rotating obstacle}
\author{Toshiaki Hishida \\
Graduate School of Mathematics, Nagoya University \\
Nagoya 464-8602, Japan \\
\texttt{hishida@math.nagoya-u.ac.jp}}
\date{}
\maketitle
\begin{abstract}
Consider the motion of a viscous incompressible fluid in a 3D
exterior domain $D$ when a rigid body $\mathbb R^3\setminus D$
moves with prescribed time-dependent translational and angular velocities.
For the linearized non-autonomous system, $L^q$-$L^r$ smoothing action
near $t=s$ as well as generation of the evolution operator
$\{T(t,s)\}_{t\geq s\geq 0}$ was shown
by Hansel and Rhandi \cite{HR14} under reasonable conditions.
In this paper we develop the $L^q$-$L^r$ decay estimates
of the evolution operator $T(t,s)$ as $(t-s)\to\infty$ and 
then apply them to the Navier-Stokes initial value problem.


\end{abstract}

\section{Introduction}
\label{intro}

Let us consider the $3$-dimensional Navier-Stokes flow past an obstacle,
which is a moving rigid body with prescribed translational and angular velocities.
In the reference frame attached to the obstacle,
the system is reduced to
\begin{equation}
\partial_tu+u\cdot\nabla u=\Delta u+(\eta+\omega\times x)\cdot\nabla u
-\omega\times u-\nabla p, \qquad
\mbox{div $u$}=0,
\label{NS}
\end{equation}
in a fixed exterior domain $D\subset \mathbb R^3$
(see Galdi \cite{Ga02} for details),
where $u=(u_1(x,t),u_2(x,t),u_3(x,t))^\top$ and $p=p(x,t)$ are unknown
velocity and pressure of a viscous incompressible fluid,
while $\eta=(\eta_1(t),\eta_2(t),\eta_3(t))^\top$ and
$\omega=(\omega_1(t),\omega_2(t),\omega_3(t))^\top$ stand for the
translational and angular velocities, respectively, of the obstacle.
Here and hereafter, $(\cdot)^\top$ denotes the transpose and all vectors
are column ones.
As usual, no-slip condition
\begin{equation}
u|_{\partial D}=\eta+\omega\times x
\label{noslip}
\end{equation}
is imposed at the boundary $\partial D$ of the obstacle, where
the boundary $\partial D$ is assumed to be of class $C^{1,1}$,
and the fluid is at rest at infinity:
\begin{equation}
u\to 0\quad\mbox{as $|x|\to\infty$}.
\label{sp-infty}
\end{equation}

The issue we are going to address in this paper is how we analyze the
case in which both $\eta(t)$ and $\omega(t)$
are actually time-dependent.
Suppose that they are locally H\"older continuous on $[0,\infty)$.
Then the problem \eqref{NS}--\eqref{sp-infty} subject to the initial condition
\begin{equation}
u(\cdot,0)=u_0
\label{IC}
\end{equation}
admits a unique solution locally in time provided the initial velocity $u_0$
is taken from $L^q(D)$ with $q\in [3,\infty)$
and fulfills the compatibility condition on the normal
trace at the boundary, that is,
$\nu\cdot (u_0-\eta(0)-\omega(0)\times x)|_{\partial D}=0$,
as well as $\mbox{div $u_0$}=0$,
where $\nu$ denotes the outer unit normal to $\partial D$.
This local existence theorem was proved by Hansel and Rhandi \cite{HR14}.
Global existence of a unique solution in the non-autonomous setting
has still remained open even if the data are small enough,
while weak solutions (in the sense of Leray-Hopf) were constructed globally
in time by Borchers \cite{Bor}.
The essential contribution of \cite{HR14} is
not only to construct the evolution operator
$\{T(t,s)\}_{t\geq s\geq 0}$ on $L^q_\sigma(D)$,
the space of solenoidal $L^q$-vector fields ($1<q<\infty$) 
with vanishing normal trace at $\partial D$,
which provides a solution operator
to the initial value problem for the linearized system
\begin{equation}
\begin{split}
&\partial_tu=\Delta u+(\eta+\omega\times x)\cdot\nabla u
-\omega\times u-\nabla p, \\
&\mbox{div $u$}=0, \\
&u|_{\partial D}=0, \\
&u\to 0\quad\mbox{as $|x|\to\infty$}, \\
&u(\cdot,s)=f,
\end{split}
\label{linearized}
\end{equation}
in $D\times [s,\infty)$, where $s\geq 0$ is the given initial time,
but also to show the $L^q$-$L^r$
smoothing action ($1<q\leq r<\infty$) near the initial time, namely,
\begin{equation}
\|T(t,s)f\|_r\leq C(t-s)^{-(3/q-3/r)/2}\|f\|_q,
\label{LqLr}
\end{equation}
\begin{equation}
\|\nabla T(t,s)f\|_r\leq C(t-s)^{-(3/q-3/r)/2-1/2}\|f\|_q,
\label{LqLr-grad}
\end{equation}
for $0\leq s<t\leq {\cal T}$ and $f\in L^q_\sigma(D)$ 
with some constant $C=C({\cal T})>0$,
where ${\cal T}\in (0,\infty)$ is arbitrarily fixed
and $\|\cdot\|_q$ denotes the norm of the space $L^q(D)$.
It should be emphasized that those results are
very nontrivial because the semigroup generated by the corresponding
autonomous operator (\cite{Shi08}, \cite{Shi10})
is never analytic (unless $\omega =0$)
so that the Tanabe-Sobolevskii theory of evolution operators of parabolic type
(see for instance \cite{T}, \cite{Pa}, \cite{Lu}) is no longer applicable.
This difficulty stems from the drift term
$(\omega\times x)\cdot\nabla u$ with unbounded coefficient,
which brings the spectrum of the linearized operator to the
imaginary axis even at large distance from the origin in the complex plane
(\cite{FaNN}, \cite{FaN07}, \cite{FaN08}, \cite{FaN10}).
In spite of this hyperbolic aspect, one may believe that
the linearized system \eqref{linearized} itself
is almost parabolic as PDE; in fact, it exhibits a certain smoothing
effect together with the singular behavior near the initial time.
For the autonomous case
(in which $\omega\in\mathbb R^3\setminus\{0\}$ is a constant vector),
this was observed first by the present author \cite{Hi99-1}, \cite{Hi99-2}
(see also \cite{Hi01} even for the specific non-autonomous case)
within the framework of $L^2$ and, later on, 
by Geissert, Heck and Hieber \cite{GHH} within the one of $L^q$.
Thus the result of Hansel and Rhandi \cite{HR14} may be regarded as
a desired generalization of \cite{GHH} to the non-autonomous case.
What is remarkable is that they constructed the evolution operator in their own way
without relying on any theory of abstract evolution equations
although the idea of iteration is somewhat similar to the one
in the Tanabe-Sobolevskii theory mentioned above.

The purpose of the present paper is to deduce the large
time behavior of the evolution operator $T(t,s)$ constructed by
Hansel and Rhandi \cite{HR14}, that is, 
estimate \eqref{LqLr} for all $t>s\geq 0$ and $f\in L^q_\sigma(D)$
with some constant $C>0$ independent of $s$ and $t$,
where $1<q\leq r<\infty$, see Theorem \ref{main} in the next section.
This plays a crucial role in studies of large time behavior
as well as global existence of small solutions to the initial value
problem \eqref{NS}--\eqref{IC}.
Our conditions on the translational and angular velocities are
\begin{equation}
\eta,\, \omega\in C^\theta([0,\infty); \mathbb R^3)
\cap L^\infty(0,\infty; \mathbb R^3)
\label{body}
\end{equation}
with some $\theta\in (0,1)$, which seem to be reasonable.
Generally speaking, it is not an easy task to investigate the asymptotic
behavior of the evolution operator for $(t-s)\to\infty$ especially
over unbounded spatial domains with boundaries such as
exterior domains.

When \eqref{linearized} is autonomous,
$L^q$-$L^r$ decay estimates of the semigroup,
not only \eqref{LqLr} (including even the case $r=\infty$)
but also \eqref{LqLr-grad} for $1<q\leq r\leq n$,
where $n$ denotes the space dimension and
$(3/q-3/r)/2$ should be replaced by $(n/q-n/r)/2$,
were established in the following literature
(among them, \cite{MSol}, \cite{DS99-1}, \cite{DS99-2} and \cite{Hi16}
studied more involved 2D case):

$\bullet\;\eta=0,\,\omega=0$ (Stokes semigroup)
\cite{I}, 
\cite{GS}, 
\cite{BM}, 
\cite{C}, 
\cite{MSol}, 
\cite{DS99-1}, \cite{DS99-2}; 

$\bullet\;\eta\neq 0,\,\omega=0$ (Oseen semigroup)
\cite{KS}, 
\cite{ES04}, \cite{ES05}, 
\cite{Hi16}; 

$\bullet\;\omega\neq 0,\,\eta=0$ (Stokes semigroup with rotating effect)
\cite{HiS}. 

\noindent
One of the effectual methods developed there 
(especially in \cite{I}, \cite{DS99-1}, \cite{KS}, 
\cite{ES04}, \cite{HiS}, \cite{Hi16})
is spectral analysis,
which is based on the principle that the regularity of the resolvent
near $\lambda =0$ implies the decay of the semigroup for $t\to\infty$,
where $\lambda$ stands for the resolvent parameter, however,
one does not have enough knowledge about such
analysis for the non-autonomous case (except the time-periodic case).
We will thus employ rather elementary approach with some contrivance.
In this case several energy estimates for derivatives
of solutions could help us
(indeed they are not enough but would clue us to the end
as in \cite{MSol}), however, those estimates do not seem to be
very useful for \eqref{linearized} except the first energy relation,
see Remark \ref{2nd-en}.
The only fine knowledge would be the $L^q$-$L^r$ decay estimate of the solution
to the same system in the whole space $\mathbb R^3$ (Lemma \ref{wh-est}).
Therefore, it is reasonable to decompose the solution
$u(t)=T(t,s)f$ of \eqref{linearized} into a suitable modification
of the associated flow in the whole space and the remaining part $v(t)$.
Then our main task is to derive the $L^r$-boundedness
of $v(t)$ uniformly in time for $r\in (2,\infty)$ by duality argument
with use of the first energy relation of the backward adjoint system.
Once we have that for some $r_0\in (2,\infty)$,
we are led to the $L^{r_0^\prime}$-boundedness of the adjoint evolution
operator $T(t,s)^*$,
where $1/r_0^\prime+1/r_0=1$, from which with the aid of the energy relation above
we obtain the $L^q$-$L^r$ estimate of
$T(t,s)^*$ for $r_0^\prime\leq q\leq r\leq 2$ (Lemma \ref{auxi}).
This argument itself is similar to the one adopted by
Maremonti and Solonnikov \cite{MSol} for the Stokes semigroup.
But, differently from theirs, we are at the beginning forced to take
the exponent $r_0$ close to $2$
because of less information about the evolution operator
itself (specifically, lack of useful information about energy estimates
for derivatives as mentioned above).
The idea is to repeat the argument above with better information at hand,
that is, the aforementioned $L^q$-$L^2$ estimate of $T(t,s)^*$
although $q\in (1,2)$ must be taken close to $2$.
We then deduce the $L^{r_0}$-boundedness of $v(t)$ for some $r_0\in (2,\infty)$
larger than before.
Such a sort of bootstrap argument eventually leads to the $L^q$-$L^r$
estimate \eqref{LqLr} with $2\leq q\leq r<\infty$ for all $t>s\geq 0$
as well as the estimate of the adjoint
\begin{equation}
\|T(t,s)^*g\|_r\leq C(t-s)^{-(3/q-3/r)/2}\|g\|_q,
\label{LqLr-adj}
\end{equation}
for all $t>s\geq 0$ and $g\in L^q_\sigma(D)$,
where $1<q\leq r\leq 2$.
In this way, estimates \eqref{LqLr} and \eqref{LqLr-adj}
are discussed simultaneously throughout the proof.
It is also possible to carry out the same procedure in which 
$T(t,s)$ and its adjoint $T(t,s)^*$ are replaced each other,
so that we obtain \eqref{LqLr} with $1<q\leq r\leq 2$
as well as \eqref{LqLr-adj} with $2\leq q\leq r<\infty$.
As a consequence, by the semigroup property of the evolution operator,
both \eqref{LqLr} and \eqref{LqLr-adj} are proved
for all $t>s\geq 0$ whenever $1<q\leq r<\infty$.

The problem under consideration is physically relevant in 2D as well,
however, our approach does not work in this case, see Remark \ref{2dim}.
Neither does it for deduction of
\eqref{LqLr-grad} for all $t>s\geq 0$ even in 3D, where $1<q\leq r\leq 3$;
indeed, it turns out that $\nabla T(t,s)f$ decays at least at the same rate
as \eqref{LqLr}, see Remark \ref{grad-pointwise},
but of course this weak decay property is useless.
The optimal gradient estimate \eqref{LqLr-grad} for all $t>s\geq 0$
is needed to solve the Navier-Stokes initial value
problem \eqref{NS}--\eqref{IC} globally in time 
as long as one follows the standard way as in Kato \cite{Ka}.
Nevertheless, in this paper, we propose another way for construction
of a unique global solution without using any pointwise decay of the gradient
of the evolution operator such as \eqref{LqLr-grad},
see Theorem \ref{main-NS}.
This approach seems to be new to the best of our knowledge.
In the duality formulation of the Navier-Stokes system
in terms of the adjoint evolution operator $T(t,s)^*$
(where this formulation itself is not new,
see \cite{KoO}, \cite{BM95}, \cite{KoY} and \cite{Y}),
our idea is to combine the energy relation of the backward adjoint
system with \eqref{LqLr-adj}
so as to deduce the large time behavior comparable to the optimal
pointwise decay of $\nabla T(t,s)^*$,
see Theorem \ref{int-decay} and Lemma \ref{key}.

This paper is organized as follows.
In the next section we introduce the evolution operator and study its adjoint.
We then provide the main theorems (Theorem \ref{main} and Theorem \ref{int-decay},
where the latter is a corollary to the former).
In Section \ref{pre} we give some preparatory results.
The central part of this paper is Section \ref{proof},
which is devoted to the proof of Theorem \ref{main}.
An application to the Navier-Stokes initial value problem
\eqref{NS}--\eqref{IC} is discussed in the final section.

\section{Evolution operator and its adjoint}
\label{evolu}

{\em 2.1. Notation}

Let us begin with introducing basic notation.
Given a domain $G\subset \mathbb R^3$, $q\in [1,\infty]$ and integer $k\geq 0$,
we denote the standard Lebesgue and Sobolev spaces by
$L^q(G)$ and by $W^{k,q}(G)$.
We abbreviate the norm $\|\cdot\|_{q,G}=\|\cdot\|_{L^q(G)}$ and even
$\|\cdot\|_q=\|\cdot\|_{q,D}$, where $D$ is the exterior domain
under consideration with $\partial D\in C^{1,1}$.
Let $C_0^\infty(G)$ be the class of all $C^\infty$ functions
with compact support in $G$, then $W^{k,q}_0(G)$ stands for
the completion of $C_0^\infty(G)$ in $W^{k,q}(G)$.
By $\langle\cdot,\cdot\rangle_G$ we denote various
duality pairings over the domain $G$.
In what follows we adopt the same symbols for denoting
vector and scalar function spaces as long as there is no confusion.
Let $X$ be a Banach space.
Then ${\cal L}(X)$ denotes the Banach space consisting of all
bounded linear operators from $X$ into itself.

We also introduce the solenoidal function space.
Let $G\subset\mathbb R^3$ be one of the following domains;
the exterior domain $D$ under consideration,
a bounded domain with $C^{1,1}$-boundary $\partial G$ and
the whole space $\mathbb R^3$.
The class $C_{0,\sigma}^\infty(G)$ consists of
all divergence-free vector fields being in $C_0^\infty(G)$.
Let $1<q<\infty$.
The space $L^q_\sigma(G)$ denotes the completion of $C_{0,\sigma}^\infty(G)$
in $L^q(G)$.
Then it is characterized as
\[
L^q_\sigma(G)=\{u\in L^q(G);\,\mbox{div $u$}=0,\, \nu\cdot u|_{\partial G}=0\},
\]
where $\nu$ stands for the outer unit normal to $\partial G$
and $\nu\cdot u$ is understood in the sense of normal trace on $\partial G$
(this boundary condition is absent when $G=\mathbb R^3$).
The space of $L^q$-vector fields admits the Helmholtz decomposition
\[
L^q(G)=L^q_\sigma(G)\oplus
\{\nabla p\in L^q(G);\, p\in L^q_{loc}(\overline{G})\},
\]
which was proved by Fujiwara and Morimoto \cite{FM},
Miyakawa \cite{Mi} and Simader and Sohr \cite{SiS}.
By $P_G=P_{G,q}: L^q(G)\to L^q_\sigma(G)$,
we denote the Fujita-Kato projection
associated with the decompostion above.
Note the duality relation $(P_{G,q})^*=P_{G,q^\prime}$,
where $1/q^\prime+1/q=1$.
We simply write $P=P_D$ for the exterior domain $D$ under consideration.
Finally, we denote several positive constants by $C$, which may
change from line to line.
\bigskip

\noindent
{\em 2.2. Evolution operator}

Suppose that
\begin{equation}
\eta,\, \omega\in C^\theta_{loc}([0,\infty);\,\mathbb R^3)
\label{body-0}
\end{equation}
for some $\theta\in (0,1)$.
We set
\begin{equation*}
\begin{split}
&|(\eta,\omega)|_{0,{\cal T}}
=\sup_{0\leq t\leq {\cal T}}\big(|\eta(t)|+|\omega(t)|\big),  \\
&|(\eta,\omega)|_{\theta,{\cal T}}=\sup_{0\leq s<t\leq {\cal T}}
\frac{|\eta(t)-\eta(s)|+|\omega(t)-\omega(s)|}{(t-s)^\theta},
\end{split}
\end{equation*}
for ${\cal T}\in (0,\infty)$.
Let $1<q<\infty$, then the linear operator
$L_\pm(t)$ relating to the exterior problem \eqref{linearized} is defined by
\begin{equation}
\begin{split}
&D_q(L_\pm(t))=\{u\in L^q_\sigma(D)\cap W_0^{1,q}(D)\cap W^{2,q}(D);\,
(\omega(t)\times x)\cdot\nabla u\in L^q(D)\}, \\
&L_\pm(t)u=-P[\Delta u\pm(\eta(t)+\omega(t)\times x)\cdot\nabla u
\mp \omega(t)\times u].
\end{split}
\label{generator}
\end{equation}
Indeed, \eqref{linearized} is reduced to the initial value problem
\begin{equation}
\partial_t u(t)+L_+(t)u(t)=0, \qquad t\in [s,\infty); \qquad u(s)=f,
\label{evo-linear}
\end{equation}
in $L^q_\sigma(D)$.
Since the domain $D_q(L_\pm(t))$ is dependent of $t$,
the space
\begin{equation}
Y_q(D):=\{u\in L^q_\sigma(D)\cap W_0^{1,q}(D)\cap W^{2,q}(D);\,
|x|\nabla u\in L^q(D)\},
\label{reg-sp}
\end{equation}
which is contained in
$D_q(L_\pm(t))$ for every $t$,
plays a role, see \cite{HR14}, in which
Hansel and Rhandi proved the following proposition.
\begin{proposition}
Suppose that $\eta$ and $\omega$ fulfill \eqref{body-0}
for some $\theta\in (0,1)$.
Let $1<q<\infty$.
The operator family
$\{L_+(t)\}_{t\geq 0}$ generates an evolution operator
$\{T(t,s)\}_{t\geq s\geq 0}$
on $L^q_\sigma(D)$ such that
$T(t,s)$ is a bounded linear operator from $L^q_\sigma(D)$
into itself with the semigroup property
\begin{equation}
T(t,\tau)T(\tau,s)=T(t,s)\qquad (t\geq \tau\geq s\geq 0);
\qquad T(s,s)=I,
\label{semi}
\end{equation}
in ${\cal L}(L^q_\sigma(D))$ and that the map
\[
\{t\geq s\geq 0\}\ni (t,s)\mapsto
T(t,s)f\in L^q_\sigma(D)
\]
is continuous for every $f\in L^q_\sigma(D)$.
Furthermore, we have the following properties.

\begin{enumerate}
\item

Let $q\leq r<\infty$.
For each ${\cal T}\in (0,\infty)$ and $m\in (0,\infty)$,
there is a constant $C=C({\cal T},m,q,r,\theta,D)>0$ such that
\eqref{LqLr} and \eqref{LqLr-grad} hold for all $(t,s)$ with
$0\leq s<t\leq {\cal T}$ and $f\in L^q_\sigma(D)$ whenever
\[
|(\eta,\omega)|_{0,{\cal T}}+|(\eta,\omega)|_{\theta,{\cal T}}\leq m.
\]
Furthermore, we have
\begin{equation}
\lim_{t\to s}\, (t-s)^{(3/q-3/r)/2+j/2}\|\nabla^j T(t,s)f\|_r=0
\label{short}
\end{equation}
for all $f\in L^q_\sigma(D)$ and $j=0,1$, except when $j=0,\, r=q$.

\item
Fix $s\geq 0$.
For every $f\in Y_q(D)$ and $t\in [s,\infty)$, we have
$T(t,s)f\in Y_q(D)$ and
\begin{equation}
T(\cdot,s)f\in C^1([s,\infty); L^q_\sigma(D))
\label{evo-cl}
\end{equation}
with
\begin{equation}
\partial_t T(t,s)f+L_+(t)T(t,s)f=0,
\qquad t\in [s,\infty),
\label{evo-1}
\end{equation}
in $L^q_\sigma(D)$.

\item
Fix $t\geq 0$.
For every $f\in Y_q(D)$, we have
\begin{equation}
T(t,\cdot)f\in C^1([0,t]; L^q_\sigma(D))
\label{evo-cl-2}
\end{equation}
with
\begin{equation}
\partial_s T(t,s)f=T(t,s)L_+(s)f,
\qquad s\in [0,t],
\label{evo-2}
\end{equation}
in $L^q_\sigma(D)$.

\end{enumerate}
\label{HR-thm}
\end{proposition}

We take $R_0>0$ satisfying
\begin{equation}
\mathbb R^3\setminus D\subset B_{R_0}:=\{x\in\mathbb R^3;\,|x|<R_0\}
\label{cov}
\end{equation}
and fix $\zeta\in C^\infty([0,\infty))$ such that
$\zeta(\rho)=1$ for $\rho\leq 1$ and
$\zeta(\rho)=0$ for $\rho\geq 2$.
We set $\phi_R(x)=\zeta(|x|/R)$ for $R\in [R_0,\infty)$, then
$\nabla \phi_R(x)=\zeta^\prime(|x|/R)\,x/(R|x|)$.
Let $f\in D_q(L_\pm(t))$ and $g\in D_{q^\prime}(L_\mp(t))$, where
$1/q^\prime +1/q=1$, then we have
\begin{equation*}
\begin{split}
&\quad \int_D[\{(\eta+\omega\times x)\cdot\nabla f\}\cdot g
+f\cdot\{(\eta+\omega\times x)\cdot\nabla g\}]\,\phi_R\,dx  \\
&=-\int_{R<|x|<2R}(f\cdot g)(\eta+\omega\times x)\cdot\nabla\phi_R\,dx.
\end{split}
\end{equation*}
Since $(\omega\times x)\cdot\nabla\phi_R=0$ and since 
$f\cdot g\in L^1(D)$, passing to the limit as $R\to\infty$ yields
\[
\langle (\eta+\omega\times x)\cdot\nabla f,\, g\rangle_D
+\langle f,\, (\eta+\omega\times x)\cdot\nabla g\rangle_D=0,
\]
which means that the non-autonomous terms of $L_\pm(t)$ are skew-symmetric.
We thus obtain
\begin{equation}
\langle L_\pm(t)f, g\rangle_D
=\langle f, L_\mp(t)g\rangle_D
\label{dual-1}
\end{equation}
for all 
$f\in D_q(L_\pm(t))$ and $g\in D_{q^\prime}(L_\mp(t))$.
If in particular $q=2$, then we find
\begin{equation}
\langle L_\pm(t)f, f\rangle_D=\|\nabla f\|_2^2
\label{coercive}
\end{equation}
for all $f\in D_2(L_\pm(t))$, which together with \eqref{evo-1} implies
the energy equality
\begin{equation}
\frac{1}{2}\,\partial_t\|T(t,s)f\|_2^2+\|\nabla T(t,s)f\|_2^2=0
\label{energy-diff}
\end{equation}
and its integral form
\begin{equation}
\frac{1}{2}\|T(t,s)f\|_2^2
+\int_\tau^t\|\nabla T(\sigma,s)f\|_2^2 \,d\sigma
=\frac{1}{2}\|T(\tau,s)f\|_2^2
\label{energy}
\end{equation}
for all $f\in Y_2(D)$ and $t\geq\tau\geq s\geq 0$.
\bigskip

\noindent
{\em 2.3. Adjoint evolution operator}

Let us fix $t\geq 0$.
The adjoint evolution operator must be related to the backward system subject
to the final condition at $t$, that is,
\begin{equation}
-\partial_s v(s)+L_-(s)v(s)=0, \qquad s\in [0,t]; \qquad v(t)=g,
\label{backward}
\end{equation}
in $L^q_\sigma(D)$, as we will explain.
It follows from the argument of \cite{HR14}
that the operator family
$\{L_-(t-\tau)\}_{\tau\in [0,t]}$
also generates an evolution operator on $L^q_\sigma(D)$,
which we denote by
$\{\widetilde T(\tau,s;\,t)\}_{0\leq s\leq\tau\leq t}$,
with the same properties as 
in Proposition \ref{HR-thm} for $T(t,s)$.
In particular, for every $g\in Y_q(D)$, we see that
$w(\tau):=\widetilde T(\tau,0;\,t)g$ solves the initial value problem
\begin{equation}
\partial_\tau w(\tau)+L_-(t-\tau)w(\tau)=0, \qquad \tau\in [0,t]; \qquad w(0)=g,
\label{adj-auxi}
\end{equation}
in $L^q_\sigma(D)$.
We set
\begin{equation}
S(t,s):=\widetilde T(t-s,0;\,t) \qquad (t\geq s\geq 0),
\label{adjoint}
\end{equation}
then, for every $g\in Y_q(D)$ and $s\in [0,t]$, we have
$S(t,s)g\in Y_q(D)$ and
\begin{equation}
S(t,\cdot)g\in C^1([0,t]; L^q_\sigma(D))
\label{adj-reg}
\end{equation}
with
\begin{equation}
\partial_s S(t,s)g=L_-(s)S(t,s)g, \qquad s\in [0,t],
\label{evo-adj-1}
\end{equation}
in $L^q_\sigma(D)$.
Namely, given $g\in Y_q(D)$,
\begin{equation}
v(s):=S(t,s)g=w(t-s)
\label{adj-2}
\end{equation}
provides a solution to \eqref{backward}.
Furthermore, $S(t,s)$ enjoys the same $L^q$-$L^r$ smoothing action
($1<q\leq r<\infty$) near the final time as in 
\eqref{LqLr}--\eqref{LqLr-grad} for $0\leq s<t\leq {\cal T}$
with some constant $C=C({\cal T})>0$,
where ${\cal T}\in (0,\infty)$ is arbitrary.
As for the energy relation to \eqref{backward}, one uses \eqref{coercive}
to get
\begin{equation}
\frac{1}{2}\,\partial_s\|S(t,s)g\|_2^2=\|\nabla S(t,s)g\|_2^2
\label{energy-adj-diff}
\end{equation}
and its integral form
\begin{equation}
\frac{1}{2}\|S(t,s)g\|_2^2+\int_s^\tau\|\nabla S(t,\sigma)g\|_2^2\,d\sigma
=\frac{1}{2}\|S(t,\tau)g\|_2^2
\label{energy-adj}
\end{equation}
for all $g\in Y_2(D)$ and $t\geq\tau\geq s\geq 0$.

We will show the following lemma.
\begin{lemma}
Let $1<q<\infty$.
Under the same conditions as in Proposition \ref{HR-thm}, we
have the duality relation 
\[
T(t,s)^*=S(t,s), \qquad S(t,s)^*=T(t,s),
\]
in ${\cal L}(L^q_\sigma(D))$ for $t\geq s\geq 0$.
\label{dual-2}
\end{lemma}

\begin{proof}
We fix $s$ and $t$ as above.
Let $f\in Y_{q^\prime}(D)$ and $g\in Y_q(D)$, where
$1/q^\prime +1/q=1$.
By virtue of \eqref{evo-1}, \eqref{dual-1} and \eqref{evo-adj-1} we observe
\begin{equation*}
\begin{split}
&\quad \partial_\tau\langle T(\tau,s)f,\, S(t,\tau)g\rangle_D  \\
&=\langle -L_+(\tau)T(\tau,s)f,\, S(t,\tau)g\rangle_D
+\langle T(\tau,s)f,\, L_-(\tau)S(t,\tau)g\rangle_D =0
\end{split}
\end{equation*}
for $\tau\in [s,t]$.
This implies that
\[
\langle T(t,s)f,\, g\rangle_D
=\langle f,\, S(t,s)g\rangle_D
\]
for $f\in Y_{q^\prime}(D)$ and $g\in Y_q(D)$;
by continuity, we have the same relation
for all $f\in L^{q^\prime}_\sigma(D)$ and $g\in L^q_\sigma(D)$ since
$Y_q(D)$ is dense in $L^q_\sigma(D)$.
This concludes $T(t,s)^*=S(t,s)$ in ${\cal L}(L^q_\sigma(D))$ and
$S(t,s)^*=T(t,s)$ in ${\cal L}(L^{q^\prime}_\sigma(D))$.
\end{proof}

Lemma \ref{dual-2} leads to the following corollary.
\begin{corollary}
Let $1<q<\infty$ and $1/q^\prime+1/q=1$.
Under the same conditions as in Proposition \ref{HR-thm},
we have the following.

\begin{enumerate}
\item
The backward semigroup property
\begin{equation}
S(\tau,s)S(t,\tau)=S(t,s) \qquad (t\geq\tau\geq s\geq 0); \qquad
S(t,t)=I,
\label{b-semi}
\end{equation}
holds in ${\cal L}(L^q_\sigma(D))$.

\item
Fix $s\geq 0$.
For every $f\in Y_{q^\prime}(D)$ and $g\in Y_q(D)$ the map
\[
[s,\infty)\ni t\mapsto \langle f,\, S(t,s)g\rangle_D
\]
is differentiable and
\begin{equation}
\partial_t \langle f,\, S(t,s)g\rangle_D
+\langle f,\, S(t,s)L_-(t)g\rangle_D=0, \qquad t\in [s,\infty).
\label{evo-adj-2}
\end{equation}
\end{enumerate}
\label{cor-adj}
\end{corollary}

\begin{proof}
The first assertion follows from \eqref{semi} and Lemma \ref{dual-2}.
Let $f\in Y_{q^\prime}(D)$ and $g\in Y_q(D)$.
Lemma \ref{dual-2} together with \eqref{dual-1} then implies
\begin{equation*}
\begin{split}
&\quad \left\langle f,\; \frac{S(t+h,s)g-S(t,s)g}{h}+S(t,s)L_-(t)g \right\rangle_D \\
&=\left\langle \frac{T(t+h,s)f-T(t,s)f}{h}+L_+(t)T(t,s)f,\; g \right\rangle_D
\end{split}
\end{equation*}
for $t,\, t+h\in [s,\infty)$.
Passing to the limit as $h\to 0$ leads to the second assertion
on account of \eqref{evo-cl}--\eqref{evo-1}.
\end{proof}
\medskip

\noindent
{\em 2.4. Main results}

We are in a position to give our main results on decay properties
of both $T(t,s)$ and $T(t,s)^*$
when further conditions \eqref{body} for some $\theta\in (0,1)$
are imposed on the translational and angular velocities.
Set
\begin{equation*}
\begin{split}
&|(\eta,\omega)|_0
=\sup_{t\geq 0}\big(|\eta(t)|+|\omega(t)|\big),  \\
&|(\eta,\omega)|_\theta
=\sup_{t>s\geq 0}\frac{|\eta(t)-\eta(s)|+|\omega(t)-\omega(s)|}{(t-s)^\theta}.
\end{split}
\end{equation*}
\begin{theorem}
Suppose that $\eta$ and $\omega$ fulfill \eqref{body} for some $\theta\in (0,1)$.
Let $1<q\leq r<\infty$.
For each $m\in (0,\infty)$, there is a constant
$C=C(m,q,r,\theta,D)>0$ such that both \eqref{LqLr} and \eqref{LqLr-adj}
hold for all $t>s\geq 0$ and $f,\, g\in L^q_\sigma(D)$ whenever
\begin{equation}
|(\eta,\omega)|_0+|(\eta,\omega)|_\theta\leq m.
\label{mag}
\end{equation}
\label{main}
\end{theorem}

Theorem \ref{main} combined with \eqref{energy} or \eqref{energy-adj}
at once yields the following estimates
\eqref{grad-decay} for $f,\, g\in Y_2(D)\cap L^q_\sigma(D)$ with $q\in (1,2]$
and, therefore, for those in $L^q_\sigma(D)$ by an approximation procedure.
\begin{theorem}
Let $q\in (1,2]$.
Under the same conditions as in Theorem \ref{main},
there is a constant $C=C(m,q,\theta,D)>0$ such that
\begin{equation}
\begin{split}
\int_t^\infty\|\nabla T(\sigma,s)f\|_2^2\,d\sigma
&\leq C(t-s)^{-3/q+3/2}\|f\|_q^2,  \\
\int_0^s\|\nabla T(t,\sigma)^*g\|_2^2\,d\sigma
&\leq C(t-s)^{-3/q+3/2}\|g\|_q^2,
\end{split}
\label{grad-decay}
\end{equation}
for all $t>s\geq 0$ and $f,\,g\in L^q_\sigma(D)$
whenever \eqref{mag} is satisfied.
\label{int-decay}
\end{theorem}
\begin{remark}
Let $1<q\leq r<\infty$.
Combining Theorem \ref{main} with Proposition \ref{unif-loc} below
gives the pointwise decay property with slow rate such as
\[
\|\nabla T(t,s)f\|_r
\leq C\|T(t-1,s)f\|_r
\leq C(t-s)^{-(3/q-3/r)/2}\|f\|_q
\]
for $t-s>2$, however, the sharp one
$(t-s)^{-\alpha}$ still remains open, where
$\alpha=\min\{(3/q-3/r)/2+1/2,\,3/2q\}$
in view of the result on the Stokes semigroup,
see \cite{MSol}, \cite{DKS} and \cite{Hi11}.
It should be noted that estimates \eqref{grad-decay} of the integral form
are comparable to the sharp pointwise decay property with $r=2$
and thus can be a substitution.
In this paper we employ \eqref{grad-decay} 
as well as Theorem \ref{main} to solve the Navier-Stokes system.
\label{grad-pointwise}
\end{remark}
\begin{remark}
It does not seem to be easy to deduce usuful higher energy estimates.
For instance, we have the second energy relation of the form
\begin{equation*}
\partial_t\|\nabla T(t,s)f\|_2^2  
+\|P\Delta T(t,s)f\|_2^2 
\leq C\big(|\omega(t)|+|\omega(t)|^2+|\eta(t)|^2\big)\|\nabla T(t,s)f\|_2^2
\end{equation*}
which is essentially due to Galdi and Silvestre \cite{GaS05}, however,
this is not enough to find decay estimates under \eqref{body}.
Even for the Oseen semigroup we had not known the decay estimate
$\|\nabla u(t)\|_2\leq C(t-s)^{-1/2}\|f\|_2$
soley by the energy method until Kobayashi and Shibata \cite{KS}
succeeded in spectral analysis to obtain $L^q$-estimate
$\|\nabla u(t)\|_q\leq C(t-s)^{-1/2}\|f\|_q$ for every $q\in (1,3]$,
where $u(t)$ denotes the solution to \eqref{linearized}
with constant $\eta\neq 0$ and $\omega=0$.
Toward analysis of the non-autonomous system under consideration,
it would be worth while trying to provide 
another proof of the gradient estimate above
without relying on spectral analysis, and one could start with
the autonomous Oseen system in the half-space $\mathbb R^3_+$,
in which the tangential derivative $\partial_{x_k} u(t)$ for
$k=1,2$ possesses the same dissipative structure as in the first energy.
\label{2nd-en}
\end{remark}

\section{Preliminaries}
\label{pre}

{\em 3.1. Uniform estimate in $t-s$}

We assume that the translational and angular velocities
satisfy \eqref{body}.
We then deduce more about the constant $C>0$ in \eqref{LqLr}
near $t=s$ than shown by Hansel and Rhandi \cite{HR14}.
They took a constant $C=C({\cal T})$ uniformly in $(t,s)$
with $0\leq s<t\leq {\cal T}$,
but it is not clear whether it can be taken uniformly in the difference
$t-s$.
For the proof of Theorem \ref{main}
we need this information, which is the issue of the following proposition.
\begin{proposition}
Suppose that $\eta$ and $\omega$ fulfill \eqref{body} for some $\theta\in (0,1)$.
Let $1<q\leq r<\infty$.
For each $\tau_*\in (0,\infty)$ and $m\in (0,\infty)$,
there is a constant $C=C(\tau_*,m,q,r,\theta,D)>0$ such that \eqref{LqLr}
and \eqref{LqLr-grad} hold for all $(t,s)$ with
\[
t-s\leq\tau_* \quad\mbox{as well as $0\leq s< t$}
\]
and $f\in L^q_\sigma(D)$ whenever \eqref{mag} is satisfied.
So does the same thing concerning estimate \eqref{LqLr-adj}
for $T(t,s)^*=S(t,s)$.
\label{unif-loc}
\end{proposition}

The latter assertion for the adjoint follows from Lemma \ref{dual-2}
and the former one, which it suffices to show.
To this end, we have to enter into the details to some extent
about the construction of the evolution operator due to \cite{HR14}.
Basically the idea is to make full use of the associated evolution operator
in the whole space $\mathbb R^3$ and the one in a bounded domain
near the boundary $\partial D$.
Both are then combined well by a cut-off technique with the aid of
Lemma \ref{bogov} below.
This approach was more or less adopted in almost all literature on the
exterior problem with moving obstacles although difficulties
in each context were overcome in his/her own device of each author.
The idea of \cite{HR14} by Hansel and Rhandi is to employ a lemma
(\cite[Lemma 3.3]{HR11}, \cite[Lemma 5.2]{HR14}, see also \cite[Lemma 4.6]{GHH})
on estimate of iterated convolution.
From its proof one can see how the constant of this estimate is determined.
It then turns out that Proposition \ref{unif-loc} follows from
\eqref{LqLr-wh} and Lemma \ref{bdd-further} below
for the evolution operator in the whole space and the one in a
bounded domain near $\partial D$, respectively.
\bigskip

\noindent
{\em 3.2. Whole space problem}

Let us begin with the non-autonomous system
\begin{equation}
\begin{split}
&\partial_tu=\Delta u+(\eta(t)+\omega(t)\times x)\cdot\nabla u-\omega(t)\times u
-\nabla p,  \\
&\mbox{div $u$}=0,
\end{split}
\label{linear-wh}
\end{equation}
in $\mathbb R^3\times [s,\infty)$ subject to
\begin{equation}
u\to 0 \quad\mbox{as $|x|\to\infty$}, \qquad
u(\cdot,s)=f,
\label{wh-sub}
\end{equation}
where $f\in L^q_\sigma(\mathbb R^3)$.
This was well studied first by Chen and Miyakawa \cite{CMi}
in a specific situation and, later on, by Geissert and Hansel \cite{GHa},
\cite{Ha} in a very general situation.
Since
\begin{equation}
\mbox{div $[(\eta+\omega\times x)\cdot\nabla u-\omega\times u]$}
=(\eta+\omega\times x)\cdot\nabla\mbox{div $u$}=0,
\label{sole}
\end{equation}
one may conclude $\nabla p=0$ within the class
$\nabla p\in L^q(\mathbb R^3)$.
Hence, the solution formula is obtained from the heat semigroup
\[
e^{t\Delta}f=(4\pi t)^{-3/2}e^{-|\cdot|^2/4t}*f
\]
simply by transformation of variables as follows, where $*$ stands for
convolution in spatial variable.
For every $y\in\mathbb R^3$, a unique solution to the initial
value problem
\[
\frac{d}{dt}\varphi(t)=-\omega(t)\times \varphi(t), \qquad \varphi(0)=y,
\]
is given by 
$\varphi(t)=Q(t)y$ in terms of an orthogonal matrix $Q(t)$ with $Q(0)=\mathbb I$
($3\times 3$ identity matrix).
Set
$\Phi(t,s)=Q(t)Q(s)^\top$, which is the evolution operator for the
ordinary differential equation above.
Under a suitable condition on $f$,
the solution to \eqref{linear-wh} is then explicitly described as
\begin{equation}
\begin{split}
u(x,t)&=\big(U(t,s)f\big)(x) \\
&:=\Phi(t,s)\left(e^{(t-s)\Delta}f\right)
\left(\Phi(t,s)^\top\Big(x+\int_s^t\Phi(t,\tau)\eta(\tau)d\tau\Big)\right) \\
&=\int_{\mathbb R^3}\Gamma(x,y; t,s)f(y)\,dy,
\end{split}
\label{formula-wh}
\end{equation}
where the kernel matrix is given by
\begin{equation*}
\begin{split}
&\Gamma(x,y; t,s) \\
&=\big(4\pi(t-s)\big)^{-3/2}\exp
\left(\frac{-\left|\Phi(t,s)^\top\left(x+\int_s^t\Phi(t,\tau)\eta(\tau)d\tau\right)-y
\right|^2}{4(t-s)}
\right)\Phi(t,s).
\end{split}
\end{equation*}
See \cite{CMi}, \cite{GHa} and \cite{Ha} for details,
but the representation above is related to the transformation
(see \cite{Ga02}), by which one obtains \eqref{NS} in the frame
attached to the obstacle from the system in the inertial frame.
Note that $\mbox{div$\big(U(t,s)f\big)$}=0$
as long as $f$ fulfills the compatibility condition $\mbox{div $f$}=0$.
We also consider the adjoint operator of $U(t,s)$, which is of the form
\[
\big(U(t,s)^*g\big)(y)
=\int_{\mathbb R^3}\Gamma(x,y; t,s)^\top g(x)\,dx.
\]
Given $t\geq 0$ and a suitable solenoidal vector field $g$,
the function $v(s):=U(t,s)^*g$ together with the trivial
pressure gradient $\nabla p=0$ formally solves the backward system
\begin{equation}
\begin{split}
-\partial_sv &=\Delta v-(\eta(s)+\omega(s)\times x)\cdot\nabla v
+\omega(s)\times v+\nabla p,  \\
\mbox{div $v$}&=0,
\end{split}
\label{adj-wh}
\end{equation}
in $\mathbb R^3\times [0,t]$ subject to
\begin{equation}
v\to 0 \quad\mbox{as $|x|\to\infty$}, \qquad
v(\cdot,t)=g.
\label{adj-wh-sub}
\end{equation}

Due to \cite{CMi}, \cite{GHa}, \cite{Ha}, \cite{HR11} and \cite{HR14},
we have the following lemma.
$L^q$-$L^r$ estimates \eqref{LqLr-wh} follow from those of the heat semigroup.
The regularity \eqref{adj-reg-wh} of the adjoint
is verified in the same way as in \eqref{adj-reg};
indeed, the third statament below corresponds to
\eqref{adj-reg}--\eqref{evo-adj-1} for the exterior problem.
\begin{lemma}
Suppose that $\eta$ and $\omega$
fulfill \eqref{body-0} for some $\theta\in (0,1)$.
Let $1<q<\infty$.
Then $\{U(t,s)\}_{t\geq s\geq 0}$ defines an evolution operator on
$L^q(\mathbb R^3)$ and on $L^q_\sigma(\mathbb R^3)$.
Similarly, $\{U(t,s)^*\}_{t\geq s\geq 0}$ defines a backward
evolution operator (see \eqref{b-semi})
on those spaces for every $q\in (1,\infty)$.
Furthermore, we have the following properties.

\begin{enumerate}
\item
Let $q\leq r\leq\infty$.
For every integer $j\geq 0$,
there is a constant $C_j=C_j(q,r)>0$, independent of $\eta$ and $\omega$, such that
\begin{equation}
\begin{split}
\|\nabla^j U(t,s)f\|_{r,\mathbb R^3}
&\leq C_j(t-s)^{-(3/q-3/r)/2-j/2}\|f\|_{q,\mathbb R^3}, \\
\|\nabla^j U(t,s)^*g\|_{r,\mathbb R^3}
&\leq C_j(t-s)^{-(3/q-3/r)/2-j/2}\|g\|_{q,\mathbb R^3},
\end{split}
\label{LqLr-wh}
\end{equation}
for all $t>s\geq 0$ and $f,\, g\in L^q(\mathbb R^3)$.

\item
Set
\[
Y_q(\mathbb R^3)=\{u\in L^q_\sigma(\mathbb R^3)\cap W^{2,q}(\mathbb R^3);\,
|x|\nabla u\in L^q(\mathbb R^3)\}.
\]
We fix $s\geq 0$.
For every $f\in Y_q(\mathbb R^3)$ and $t\in [s,\infty)$, we have
$U(t,s)f\in Y_q(\mathbb R^3)$ and
\begin{equation}
u:=U(\cdot,s)f\in C^1([s,\infty); L^q_\sigma(\mathbb R^3)),
\label{evo-cl-wh}
\end{equation}
which satisfies \eqref{linear-wh}--\eqref{wh-sub} in $L^q_\sigma(\mathbb R^3)$.

\item
Fix $t\geq 0$.
For every $g\in Y_q(\mathbb R^3)$ and $s\in [0,t]$, we have
$U(t,s)^*g\in Y_q(\mathbb R^3)$ and
\begin{equation}
v:=U(t,\cdot)^*g\in C^1([0,t]; L^q_\sigma(\mathbb R^3)),
\label{adj-reg-wh}
\end{equation}
which satisfies \eqref{adj-wh}--\eqref{adj-wh-sub} in
$L^q_\sigma(\mathbb R^3)$.
\end{enumerate}
\label{wh-est}
\end{lemma}
\medskip

\noindent
{\em 3.3. Interior problem}

We fix $R\in [R_0,\infty)$,
where $R_0$ is as in \eqref{cov},
and proceed to the non-autonomous system
\begin{equation}
\begin{split}
&\partial_tu=\Delta u+(\eta+\omega\times x)\cdot\nabla u
-\omega\times u-\nabla p,  \\
&\mbox{div $u$}=0, \\
\end{split}
\label{linear-bdd}
\end{equation}
in $D_R\times [s,\infty)$ subject to
\begin{equation}
u|_{\partial D_R}=0, \qquad
u(\cdot,s)=f.
\label{bdd-sub}
\end{equation}
Using the Fujita-Kato projection $P_{D_R}$
associated with the Helmholtz decomposition,
we define the Stokes operator
\[
D_q(A)=L^q_\sigma(D_R)\cap W^{1,q}_0(D_R)\cap W^{2,q}(D_R), \qquad
Au=-P_{D_R}\Delta u,
\]
and the operator
\[
D_q(L_R(t))=D_q(A), \qquad
L_R(t)=A+B(t),
\]
where the non-autonomous term
\begin{equation}
\begin{split}
B(t)u
&:=-P_{D_R}[(\eta(t)+\omega(t)\times x)\cdot\nabla u-\omega(t)\times u]  \\
&=-(\eta(t)+\omega(t)\times x)\cdot\nabla u+\omega(t)\times u
\end{split}
\label{perturb}
\end{equation}
is nothing but lower order perturbation from the Stokes operator
for the interior problem unlike the exterior problem.
The latter equality in \eqref{perturb}
holds for $u\in D_q(A)$ as shown in the proof of Lemma \ref{bdd-further} below.
For each $t\geq 0$, the operator $L_R(t)$
generates an analytic semigroup on $L^q_\sigma(D_R)$,
see the resolvent estimate \eqref{resol} below.
Under the condition \eqref{body-0} for some $\theta\in (0,1)$,
it is not difficult to apply the Tanabe-Sobolevskii theory to
the operator family
$\{L_R(t)\}_{t\geq 0}$.
It then turns out that this family generates the evolution
operator $\{V(t,s)\}_{t\geq s\geq 0}$ of parabolic type on
$L^q_\sigma(D_R)$;
indeed, this was the observation by \cite{HR14} (see also \cite{Hi01}).
In the present paper, further consideration 
under the condition \eqref{body} is needed.
\begin{lemma}
Suppose that $\eta$ and $\omega$ fulfill \eqref{body} for some $\theta\in (0,1)$.
Let $1<q\leq r<\infty$.
For each $\tau_*\in (0,\infty)$, $m\in (0,\infty)$ and $j=0,1$,
there are constants $C_j=C_j(\tau_*,m,q,r,\theta,D_R)>0$ and
$C_2=C_2(\tau_*,m,q,\theta,D_R)>0$ such that
\begin{equation}
\|\nabla^j V(t,s)f\|_{r,D_R}\leq C_j(t-s)^{-(3/q-3/r)/2-j/2}\|f\|_{q,D_R},
\label{LqLr-bdd}
\end{equation}
\begin{equation}
\|p(t)\|_{q,D_R}\leq C_2(t-s)^{-(1+1/q)/2}\|f\|_{q,D_R}
\label{press-est}
\end{equation}
for all $(t,s)$ with
$t-s\leq\tau_*$ as well as $0\leq s<t$ and 
$f\in L^q_\sigma(D_R)$ whenever \eqref{mag} is satisfied.
Here, $p(t)$ denotes the pressure to \eqref{linear-bdd}
associated with $u(t)=V(t,s)f$ and it is singled out subject to the side condition
$\int_{D_R}p(x,t)dx=0$.
\label{bdd-further}
\end{lemma}

\begin{proof}
Set
$\Sigma=\{\lambda\in\mathbb C;\,|\arg\lambda|\leq 3\pi/4\}\cup\{0\}$.
We know (\cite{FaS}, \cite{G}, \cite{Sol}) that 
$\Sigma\subset \rho(-A)$ with
\[
\|\nabla^j (\lambda +A)^{-1}\|
\leq C(1+|\lambda|)^{-(2-j)/2}
\]
for all $\lambda\in\Sigma$ and $j=0,1$,
where we fix $q\in (1,\infty)$ and abbreviate
$\|\cdot\|=\|\cdot\|_{{\cal L}(L^q_\sigma(D_R))}$.
Let $k>0$, then by $|\lambda +k|\geq k/2^{1/2}$
we have the following uniform boundedness in
$\lambda\in\Sigma$ and $t\geq 0$:
\begin{equation*}
\begin{split}
\|B(t)(\lambda +k+A)^{-1}\|
&\leq C|(\eta,\omega)|_0\sum_{j=0}^1\|\nabla^j(\lambda +k+A)^{-1}\| \\
&\leq Cm\sum_{j=0}^1 (1+k)^{-(2-j)/2}.
\end{split}
\end{equation*}
We are thus able to take $k=k(m)>0$ large enough to obtain
\[
\|B(t)(\lambda +k+A)^{-1}\|\leq\frac{1}{2}
\]
for all $\lambda\in\Sigma$ and $t\geq 0$, which yields the existence
of the bounded inverse
\[
(\lambda +k+L_R(t))^{-1}
=(\lambda +k+A)^{-1}\big[1+B(t)(\lambda +k+A)^{-1}\big]^{-1}
\]
together with
\begin{equation}
\|\nabla^j(\lambda +k+L_R(t))^{-1}\|
\leq C(1+|\lambda +k|)^{-(2-j)/2}
\leq c_*(1+|\lambda|)^{-(2-j)/2}
\label{resol}
\end{equation}
for all $\lambda\in\Sigma$, $t\geq 0$ and $j=0,\,1$,
where the constant $c_*=c_*(m)$
depends on $m$ via $k=k(m)$.
This implies that
\begin{equation}
\begin{split}
&\quad \|(k+L_R(t))(k+L_R(\tau))^{-1}-(k+L_R(s))(k+L_R(\tau))^{-1}\|  \\
&=\|(B(t)-B(s))(k+L_R(\tau))^{-1}\|  \\
&\leq C|(\eta,\omega)|_\theta\,|t-s|^\theta
\sum_{j=0}^1\|\nabla^j(k+L_R(\tau))^{-1}\|  \\
&\leq Cc_*m\,|t-s|^\theta
\end{split}
\label{hoelder}
\end{equation}
for all $t,\, s,\, \tau\geq 0$ and that
\begin{equation}
\left\|\nabla^j e^{-(t-s)(k+L_R(s))}\right\|\leq Cc_*(t-s)^{-j/2}
\label{1st-app}
\end{equation}
for all $t>s\geq 0$ and $j=0,\,1$.
Set
\begin{equation}
G_1(t,s)=-\{(k+L_R(t))-(k+L_R(s))\}e^{-(t-s)(k+L_R(s))},
\label{G-1}
\end{equation}
then \eqref{1st-app} leads to
\begin{equation}
\begin{split}
\|G_1(t,s)\|
&=\left\|(B(t)-B(s))e^{-(t-s)(k+L_R(s))}\right\|  \\
&\leq Cc_*m\big(1+\tau_*^{1/2}\big)(t-s)^{\theta-1/2}
\end{split}
\label{para}
\end{equation}
for all $t>s\geq 0$ with $t-s\leq\tau_*$.
By \eqref{resol} and \eqref{hoelder} one can provide a parametrix
of the evolution operator $V(t,s)$ along the procedure due to
Tanabe, in which the remainder part
\[
W(t,s):=e^{-k(t-s)}V(t,s)-e^{-(t-s)(k+L_R(s))}
\]
is constructed in the form
\[
W(t,s)=\int_s^t e^{-(t-\tau)(k+L_R(\tau))}G(\tau,s)\,d\tau
\]
by means of iteration
\[
G(t,s)=\sum_{j=1}^\infty G_j(t,s), \qquad
G_j(t,s)=\int_s^t G_1(t,\tau)G_{j-1}(\tau,s)\,d\tau \quad (j\geq 2)
\]
starting from $G_1(t,s)$ given by \eqref{G-1},
see \cite[Chapter 5, Section 2]{T} for details.
It follows from
\eqref{1st-app} and \eqref{para} 
(together with the H\"older estimate of
$G_1(t,s)-G_1(\tau,s)$ for $t>\tau >s\geq 0$) that
\[
\left\|e^{-k(t-s)}V(t,s)\right\|
+(t-s)\left\|\partial_t\big\{e^{-k(t-s)}V(t,s)\big\}\right\| \leq c_0
\]
with some constant $c_0=c_0(\tau_*,m,q,\theta,D_R)>0$ and thereby
\begin{equation}
\|V(t,s)\|+(t-s)\|\partial_t V(t,s)\|
\leq c_0(1+k\tau_*)e^{k\tau_*}
\label{basic}
\end{equation}
for all $t>s\geq 0$ with $t-s\leq\tau_*$.
Since
\[
\|B(t)u\|_{q,D_R}
\leq\frac{1}{2}\|Au\|_{q,D_R}+C(m+m^2)\|u\|_{q,D_R},
\]
we have
\[
\|u\|_{W^{2,q}(D_R)}
\leq C\|Au\|_{q,D_R}
\leq C\|L_R(t)u\|_{q,D_R}
+C(m+m^2)\|u\|_{q,D_R}
\]
for $t\geq 0$ and $u\in D_q(A)$,
in which we set $u=V(t,s)f$ and use \eqref{basic} to obtain
\begin{equation}
\|V(t,s)f\|_{W^{2,q}(D_R)}
\leq C(t-s)^{-1}\|f\|_{q,D_R}
\label{2nd}
\end{equation}
with some constant $C=C(\tau_*,m,q,\theta,D_R)>0$
for all $t>s\geq 0$ with $t-s\leq\tau_*$.
The Gagliardo-Nirenberg inequality and the semigroup property thus imply
\eqref{LqLr-bdd}.

Let us consider the estimate of the associated pressure by following
the idea of \cite[Section 3]{HiS}, \cite[Section 4]{HR14}.
To this end, we first verify the latter equality of \eqref{perturb}
for $u\in D_q(A)$.
On account of \eqref{sole}, it suffices to show that the normal trace
$\nu\cdot(\partial_iu)$ vanishes on the boundary 
$\partial D_R$ for each
$i\in\{1,2,3\}$, where $\nu$ stands for the outer unit normal to
$\partial D_R$.
Because $C_{0,\sigma}^\infty(D_R)$ is dense in the space
$\{u\in W^{1,q}_0(D_R); \mbox{div $u$}=0\}$,
we take $\phi_k\in C_{0,\sigma}^\infty(D_R)$ ($k=1,2,...$) satisfying
$\|\phi_k-u\|_{W^{1,q}(D_R)}\to 0$ as $k\to\infty$.
Since $\partial_i\phi_k\in C_{0,\sigma}^\infty(D_R)$,
we conclude that $\partial_i u\in L^q_\sigma(D_R)$ and, hence,
\begin{equation}
\nu\cdot (\partial_iu)|_{\partial D_R}=0, \qquad (i=1,2,3)
\label{n-tra}
\end{equation}
for every $u\in D_q(A)$.
Let $\varphi\in C_0^\infty(D_R)$ and let $\psi$ be
a solution to the Neumann problem
\[
\Delta\psi=\varphi-\frac{1}{|D_R|}\int_{D_R}\varphi(y)dy \quad\mbox{in $D_R$},
\qquad \partial_\nu\psi|_{\partial D_R}=0.
\]
By \eqref{linear-bdd}, \eqref{bdd-sub} and \eqref{n-tra} we observe
\begin{equation}
\langle \nabla p, \nabla\psi\rangle_{D_R}
=\langle\Delta u, \nabla\psi\rangle_{D_R},
\label{press-rela}
\end{equation}
where $u=V(t,s)f$.
Since $p$ is chosen so that
$\int_{D_R}p(x,t)dx=0$, we deduce from \eqref{press-rela} that
\begin{equation*}
\begin{split}
\langle p, \varphi\rangle_{D_R}
&=\langle p, \Delta\psi\rangle_{D_R}
=-\langle\Delta u, \nabla\psi\rangle_{D_R}  \\
&=\langle\nabla u, \nabla^2\psi\rangle_{D_R}
-\int_{\partial D_R}(\partial_\nu u)\cdot\nabla\psi\,d\sigma
\end{split}
\end{equation*}
for all $\varphi\in C_0^\infty(D_R)$.
We make use of the trace estimate together with
$\|\psi\|_{W^{2,q^\prime}(D_R)}\leq C\|\varphi\|_{q^\prime,D_R}$,
where $1/q^\prime+1/q=1$, to get
\[
\|p(t)\|_{q,D_R}
\leq C\|\nabla^2u(t)\|_{q,D_R}^{1/q}\|\nabla u(t)\|_{q,D_R}^{1-1/q}
+C\|\nabla u(t)\|_{q,D_R},
\]
which leads to \eqref{press-est} by virtue of
\eqref{LqLr-bdd} and \eqref{2nd}.
The proof is complete.
\end{proof}
\medskip

\noindent
{\em 3.4. Bogovskii operator}

Let $G\subset\mathbb R^n$ ($n\geq 2$)
be a bounded domain with Lipschitz boundary $\partial G$.
The boundary value problem
\[
\mbox{div $w$}=f \quad\mbox{in $G$}, \qquad
w|_{\partial G}=0,
\]
admits a lot of solutions as long as $f$ possesses an appropriate
regularity and satisfies the compatibility condition
$\int_G f(x)dx=0$.
Among them a particular solution found by Bogovskii \cite{B}
is convenient to recover the solenoidal
condition in cut-off procedures because of several fine properties of
his solution.
To be precise, we have the following lemma,
see \cite[Theorem 1]{B}, \cite[Theorem 2.4 (a)--(c)]{BS}, 
\cite[Theorem III.3.3]{Ga-b}, \cite[Theorem 2.5]{GHH-b}
and the references therein.
\begin{lemma}
Let $G\subset\mathbb R^n$, $n\geq 2$, be a bounded domain with
Lipschitz boundary.
There exists a linear operator
${\mathbb B}_G: C_0^\infty(G)\to C_0^\infty(G)^n$ with the following properties:
For every $q\in (1,\infty)$ and integer $k\geq 0$,
there is a constant $C=C(q,k,G)>0$ such that
\begin{equation}
\|\nabla^{k+1}{\mathbb B}_Gf\|_{q,G}\leq C\|\nabla^kf\|_{q,G}
\label{bog-1}
\end{equation}
and that
\begin{equation}
\mbox{div $({\mathbb B}_Gf)$}=f \qquad\mbox{if} \quad
\int_G f(x)\,dx=0,
\label{div-eq}
\end{equation}
where the constant $C$ is invariant under dilation of the domain $G$.
The operator ${\mathbb B}_G$ extends uniquely to a bounded operator
from $W^{k,q}_0(G)$ to $W^{k+1,q}_0(G)^n$ so that
\eqref{bog-1} and \eqref{div-eq} still hold true.
Furthermore, for every $q\in (1,\infty)$, it also extends uniquely to a bounded
operator from
$W^{1,q^\prime}(G)^*$ to $L^q(G)^n$, namely,
\begin{equation}
\|{\mathbb B}_Gf\|_{q,G}\leq C\|f\|_{W^{1,q^\prime}(G)^*}
\label{bog-2}
\end{equation}
with some constant $C=C(q,G)>0$, where
$1/q^\prime+1/q=1$.
\label{bogov}
\end{lemma}
\medskip

\noindent
{\em 3.5. A useful lemma}

We conclude this section with the following lemma
that is useful for both the proof of Theorem \ref{main} and analysis
of the Navier-Stokes flow.
It is not related to the evolution operator
and might be of independent interest.
The issue is how to obtain the optimal growth rate of the integral,
see \eqref{grow} below, for $t\to\infty$ from estimate of
the square integral.
\begin{lemma}
Fix $s\in\mathbb R$ and let $\alpha <1$.
Suppose that $z=z(\tau)$ is a real-valued function being in
$L^2_{loc}((s,\infty))$ and that
\begin{equation}
\int_{s+t}^{s+2t}z(\tau)^2\,d\tau
\leq Mt^{-\alpha}
\label{square}
\end{equation}
for all $t>0$ with some constant $M>0$.
Then we have
$z\in L^1_{loc}([s,\infty))$ with
\begin{equation}
\int_s^{s+t}|z(\tau)|\,d\tau\leq CM^{1/2}\,t^{(1-\alpha)/2}
\label{grow}
\end{equation}
for all $t>0$ with some constant $C=C(\alpha)>0$.
\label{growth-est}
\end{lemma}

\begin{proof}
Although the proof is quite simple, we give it for readers' convenience
(since I do not find it in literature).
By \eqref{square} and the Schwarz inequality we have
\[
\int_{s+t}^{s+2t}|z(\tau)|\,d\tau\leq M^{1/2}\,t^{(1-\alpha)/2}
\]
for all $t>0$.
We split the interval $(s,s+t)$ dyadically and then utilize the estimate above
to find
\[
\int_s^{s+t}|z(\tau)|\,d\tau
=\sum_{j=0}^\infty\int_{s+t/2^{j+1}}^{s+t/2^j} |z(\tau)|\,d\tau
\leq M^{1/2}\,t^{(1-\alpha)/2}
\sum_{j=0}^\infty\Big(2^{(\alpha-1)/2}\Big)^{j+1}
\]
for all $t>0$, which yields \eqref{grow} on account of $\alpha<1$.
\end{proof}

\section{Proof of Theorem \ref{main}}
\label{proof}

In this section we will prove Theorem \ref{main}.
Let us fix $m>0$ and assume \eqref{mag}.
We first consider \eqref{LqLr} for $2\leq q\leq r<\infty$ simultaneously with
\eqref{LqLr-adj} for $1<q\leq r\leq 2$.
We fix a cut-off function $\phi\in C_0^\infty(B_{3R_0})$ such that
$\phi=1$ on $B_{2R_0}$, where $R_0$ is fixed as in \eqref{cov}.
Given $f\in C^\infty_{0,\sigma}(D)\subset C^\infty_{0,\sigma}(\mathbb R^3)$,
we take the solution $U(t,s)f$,
see \eqref{formula-wh},
to the whole space problem \eqref{linear-wh}
with the initial velocity $f$ at the initial time $s\geq 0$.
We may regard the solution $T(t,s)f$ for the exterior problem
\eqref{linearized} as a perturbation from
a modification of $U(t,s)f$;
to be precise, let us describe $T(t,s)f$ in the form
\begin{equation}
T(t,s)f=(1-\phi)U(t,s)f
+\mathbb B\left[\big(U(t,s)f\big)\cdot\nabla\phi\right]+v(t),
\label{decompo}
\end{equation}
where the perturbation is denoted by $v(t)=v(t;s)$ and
$\mathbb B:=\mathbb B_{A_{R_0}}$ is the Bogovskii operator
on the domain $A_{R_0}:=B_{3R_0}\setminus\overline{B_{R_0}}$
given by Lemma \ref{bogov}.
It is easily seen that $v(t)$ together with the pressure $p(t)$
associated with $T(t,s)f$ obeys
\begin{equation*}
\begin{split}
&\partial_tv=\Delta v+(\eta(t)+\omega(t)\times x)\cdot\nabla v-\omega(t)\times v
-\nabla p+F,  \\
&\mbox{div $v$}=0,
\end{split}
\end{equation*}
in $D\times [s,\infty)$ subject to
\begin{equation*}
\begin{split}
&v|_{\partial D}=0, \\
&v\to 0\quad\mbox{as $|x|\to\infty$}, \\
&v(\cdot,s)=\widetilde f:=\phi f-\mathbb B[f\cdot\nabla\phi],
\end{split}
\end{equation*}
where the forcing term $F$ is given by
\begin{equation*}
\begin{split}
F(x,t)
&=-2\nabla\phi\cdot\nabla U(t,s)f
-[\Delta\phi+(\eta(t)+\omega(t)\times x)\cdot\nabla\phi]\,U(t,s)f  \\
&\quad -\mathbb B\left[\partial_t\big(U(t,s)f\big)\cdot\nabla\phi\right]
+\Delta\mathbb B\left[\big(U(t,s)f\big)\cdot\nabla\phi\right]  \\
&\quad
+(\eta(t)+\omega(t)\times x)\cdot\nabla 
\mathbb B\left[\big(U(t,s)f\big)\cdot\nabla\phi\right]  \\
&\quad
-\omega(t)\times\mathbb B\left[\big(U(t,s)f\big)\cdot\nabla\phi\right],
\end{split}
\end{equation*}
which behaves like
\begin{equation}
\|F(t)\|_q \leq
\left\{
\begin{array}{ll}
C(m+1)(t-s)^{-1/2}\|f\|_q, &0<t-s<1, \\
C(m+1)(t-s)^{-3/2q}\|f\|_q, \qquad &t-s\geq 1,
\end{array}
\right.
\label{force}
\end{equation}
for $1<q<\infty$.
One can verify \eqref{force} by virtue of \eqref{LqLr-wh} 
and \eqref{bog-1}--\eqref{bog-2}
together with the first equation (with $\nabla p=0$) of \eqref{linear-wh}.
In fact, the only term in which one needs \eqref{bog-2} is
\begin{equation*}
\begin{split}
\left\|\mathbb B
\left[\partial_t\big(U(t,s)f\big)\cdot\nabla\phi\right]\right\|_{q,A_{R_0}}
&\leq C\|\partial_t\big(U(t,s)f\big)\cdot\nabla\phi\|_{W^{1,q^\prime}(A_{R_0})^*} \\
&\leq C\|\nabla U(t,s)f\|_{q,A_{R_0}}+Cm\|U(t,s)f\|_{q,A_{R_0}}
\end{split}
\end{equation*}
and the other terms are harmless by using solely \eqref{bog-1}.

In view of \eqref{decompo} together with
\eqref{evo-cl}, \eqref{evo-cl-wh} and Lemma \ref{bogov},
we deduce from
$f\in C^\infty_{0,\sigma}(D)$ that
\[
v\in C^1([s,\infty); L^q_\sigma(D))
\]
as well as $v(t)\in Y_q(D)$ for every $q\in (1,\infty)$.
We can thus employ \eqref{evo-2} to compute
$\partial_\tau\{T(t,\tau)v(\tau)\}$, so that we are led to the Duhamel formula
\[
v(t)=T(t,s)\widetilde f +\int_s^tT(t,\tau)PF(\tau)\,d\tau
\]
in $L^q_\sigma(D)$.
It is convenient to consider the duality formulation
\begin{equation}
\langle v(t), \psi\rangle_D
=\langle\widetilde f, T(t,s)^*\psi\rangle_D
+\int_s^t\langle F(\tau), T(t,\tau)^*\psi\rangle_D\,d\tau
\label{duhamel-1}
\end{equation}
for $\psi\in C^\infty_{0,\sigma}(D)$.

Let $r\in (2,\infty)$.
We intend to prove the boundedness uniformly in large $t-s$, that is,
\begin{equation}
\|v(t)\|_r\leq C\|f\|_r \qquad\mbox{for $t-s>3$},
\label{unif-bdd}
\end{equation}
with some constant $C=C(m,r,\theta,D)>0$ independent of such $(t,s)$,
where $m$ is as in \eqref{mag}.
Once we have that, we can conclude the following decay properties.
\begin{lemma}
In addition to the conditions in Theorem \ref{main},
suppose that, with some $r_0\in (2,\infty)$, estimate
\eqref{unif-bdd} holds for all $f\in C^\infty_{0,\sigma}(D)$.

\begin{enumerate}
\item
Let $2\leq q\leq r\leq r_0$.
Then there is a constant $C=C(m,q,r,r_0,\theta,D)>0$ such that \eqref{LqLr}
holds for all $t>s\geq 0$ and $f\in L^q_\sigma(D)$.

\item
Let $r_0^\prime\leq q\leq r\leq 2$, where
$1/r_0^\prime +1/r_0=1$.
Then there is a constant $C=C(m,q,r,r_0,\theta,D)>0$ such that \eqref{LqLr-adj} 
holds for all $t>s\geq 0$ and $g\in L^q_\sigma(D)$.
\end{enumerate}
\label{auxi}
\end{lemma}

\begin{proof}
In view of \eqref{decompo}, we see from
\eqref{unif-bdd} and \eqref{LqLr-wh} with the aid of
Lemma \ref{bogov} that
\[
\|T(t,s)f\|_{r_0}\leq C\|f\|_{r_0}
\]
for $t-s>3$ and, therefore, for all $t>s\geq 0$ and
$f\in L^{r_0}_\sigma(D)$ with some constant $C=C(m,r_0,\theta,D)>0$
since we know the estimate for
$t-s\leq 3$ by Proposition \ref{unif-loc}.
By duality we have
\begin{equation}
\|T(t,s)^*g\|_q\leq C\|g\|_q
\label{unif-bdd-adj}
\end{equation}
for all $t>s\geq 0$ and $g\in L^q_\sigma(D)$ with $q=r_0^\prime\in (1,2)$
and, thereby, with $q\in [r_0^\prime,2]$
on account of contraction property \eqref{energy-adj} in $L^2$.
Hence, by embedding relation we obtain
\[
\|T(t,s)^*g\|_2
\leq\|T(t,s)^*g\|_6^\mu \|T(t,s)^*g\|_q^{1-\mu}
\leq C\|\nabla T(t,s)^*g\|_2^\mu \|g\|_q^{1-\mu}
\]
for $g\in L^q_\sigma(D)$, where $q\in [r_0^\prime,2)$ and
$1/2=\mu/6+(1-\mu)/q$.
We fix $t>0$, then the last inequality together with the energy relation
\eqref{energy-adj-diff} implies that
\[
\partial_s\|T(t,s)^*g\|_2^2
\geq C\|g\|_q^{-2(1/\mu-1)}\|T(t,s)^*g\|_2^{2/\mu}
\]
for all $s\in [0,t]$ and $g\in C^\infty_{0,\sigma}(D)\setminus\{0\}$.
By solving this differential inequality (as in \cite[Section 5]{MSol})
we conclude \eqref{LqLr-adj} when $r=2$.
Combining this with \eqref{unif-bdd-adj} leads to
\eqref{LqLr-adj} for $r_0^\prime\leq q\leq r\leq 2$.
The other estimate \eqref{LqLr} follows from \eqref{LqLr-adj} by duality.
\end{proof}

Let us derive \eqref{unif-bdd} when $r>2$.
As we will see, our main task is to estimate the integral over the interval
$(s+1,t-1)$ of the RHS of \eqref{duhamel-1};
in fact, the other terms are easily treated as follows.
By the energy relation \eqref{energy-adj} (with $\tau=t-1$) and
by Proposition \ref{unif-loc} (with $\tau_*=1$)
for $T(t,s)^*$ we obtain
\begin{equation}
\begin{split}
|\langle \widetilde f, T(t,s)^*\psi\rangle_D|
&\leq \|\widetilde f\|_{2,D_{3R_0}}\|T(t,s)^*\psi\|_2  \\
&\leq C\|f\|_r\|T(t,t-1)^*\psi\|_2  \\
&\leq C\|f\|_r\|\psi\|_{r^\prime}
\end{split}
\label{top}
\end{equation}
for $t-s>3$ and every $r\in (2,\infty)$,
where
$1/r^\prime+1/r=1$ and
$D_{3R_0}=D\cap B_{3R_0}$ since $\widetilde f=0$ outside $D_{3R_0}$.
From \eqref{force} in addition to the same reasoning as above it follows that
\begin{equation*}
\begin{split}
&\quad 
\left|
\left(\int_s^{s+1}+\int_{t-1}^t\right)\langle F(\tau), T(t,\tau)^*\psi\rangle_D\,d\tau
\right|  \\
&\leq \left(\int_s^{s+1}+\int_{t-1}^t\right)
\|F(\tau)\|_{2,A_{R_0}}\|T(t,\tau)^*\psi\|_2\,d\tau \\
&\leq C\|T(t,t-1)^*\psi\|_2 \int_s^{s+1}\|F(\tau)\|_r\,d\tau  \\
&\qquad\qquad
+C\|\psi\|_{r^\prime}\int_{t-1}^t\|F(\tau)\|_r (t-\tau)^{-(3/r^\prime-3/2)/2}\,d\tau \\
&\leq C(m+1)\|f\|_r\|\psi\|_{r^\prime}
\Big( \int_s^{s+1}(\tau -s)^{-1/2}\,d\tau  \\
&\qquad\qquad
+\int_{t-1}^t (\tau -s)^{-3/2r}(t-\tau)^{-(3/r^\prime-3/2)/2}\,d\tau \Big),
\end{split}
\end{equation*}
which yields
\begin{equation}
\left|\int_s^{s+1}+\int_{t-1}^t\right|
\leq C\|f\|_r\|\psi\|_{r^\prime}\big\{1+(t-s-1)^{-3/2r}\big\}
\label{near-end}
\end{equation}
for $t-s>3$ and every $r\in (2,\infty)$ with some constant
$C=C(m,r,\theta,D)>0$.

We turn to the integral
\[
J:=\int_{s+1}^{t-1}\langle F(\tau), T(t,\tau)^*\psi\rangle_D\,d\tau
\]
for which three steps are needed.
The details are given as follows.

Since $T(t,\tau)^*\psi=S(t,\tau)\psi\in Y_2(D)$ vanishes at the 
boundary $\partial D$,
see \eqref{reg-sp}, and thereby satisfies the Poincar\'e inequality
in the bounded domain $D_{3R_0}$, we have
\begin{equation}
\begin{split}
|J|
&\leq \int_{s+1}^{t-1}\|F(\tau)\|_{2,A_{R_0}}\|T(t,\tau)^*\psi\|_{2,D_{3R_0}}\,d\tau \\
&\leq C(m+1)\|f\|_r\int_{s+1}^{t-1}(\tau-s)^{-3/2r}\|\nabla T(t,\tau)^*\psi\|_2\,d\tau
\end{split}
\label{J}
\end{equation}
by \eqref{force} when $r>2$.
We use the energy relation \eqref{energy-adj} to find
\begin{equation}
\begin{split}
|J|
&\leq C(m+1)\|f\|_r\left(\int_{s+1}^{t-1}(\tau-s)^{-3/r}\,d\tau\right)^{1/2}
\left(\int_{s+1}^{t-1}\|\nabla T(t,\tau)^*\psi\|_2^2\,d\tau\right)^{1/2}  \\
&\leq C(m+1)\|f\|_r\|T(t,t-1)^*\psi\|_2
\left(\int_{s+1}^{t-1}(\tau-s)^{-3/r}\,d\tau\right)^{1/2}.
\end{split}
\label{J-2}
\end{equation}
By Proposition \ref{unif-loc} (with $\tau_*=1$) there exists a
constant $C=C(m,r,\theta,D)>0$ such that
\[
|J|\leq C\|f\|_r\|\psi\|_{r^\prime}
\]
for $t-s>3$ provided $2<r<3$, which combined with \eqref{near-end}
as well as \eqref{top} yields \eqref{unif-bdd} for such $r$.

Let us proceed to the next step.
By virtue of Lemma \ref{auxi} we have \eqref{LqLr-adj} with
$3/2<q\leq r\leq 2$ for all $t>s\geq 0$ as a consequence of the result above.
Let $r\in (3,\infty)$.
Given $\varepsilon >0$ arbitrarily small,
we choose $q\in (3/2,2)$ (close to $3/2$) to get
\begin{equation}
\begin{split}
\|T(t,s)^*\psi\|_2
&\leq C_\varepsilon (t-s-1)^{-1/4+\varepsilon}\|T(t,t-1)^*\psi\|_q  \\
&\leq C_\varepsilon (t-s-1)^{-1/4+\varepsilon}\|\psi\|_{r^\prime}
\end{split}
\label{weak-d1}
\end{equation}
for $t-s>1$ on account of Proposition \ref{unif-loc} 
as well as the backward semigroup property \eqref{b-semi}.
We split the integral in \eqref{J} into
\begin{equation}
\int_{s+1}^{t-1}(\tau-s)^{-3/2r}\|\nabla T(t,\tau)^*\psi\|_2\,d\tau
=\int_{s+1}^{(s+t)/2}+\int_{(s+t)/2}^{t-1}
\label{J-split}
\end{equation}
for $t-s>3$.
By use of \eqref{weak-d1} together with \eqref{energy-adj} we obtain
\begin{equation}
\begin{split}
\int_{s+1}^{(s+t)/2}\|\nabla T(t,\tau)^*\psi\|_2^2\,d\tau
&\leq\frac{1}{2}\|T(t,(s+t)/2)^*\psi\|_2^2  \\
&\leq C_\varepsilon (t-s-2)^{-1/2+2\varepsilon}\|\psi\|_{r^\prime}^2
\end{split}
\label{J-first}
\end{equation}
for $t-s>2$ when $r>3$.
Following the notation \eqref{adj-2}, we set
$w(t-\tau)=T(t,\tau)^*\psi$;
then \eqref{J-first} is rewritten as
\[
\int_{(t-s)/2}^{t-s-1}\|\nabla w(\tau)\|_2^2\,d\tau
\leq C_\varepsilon (t-s-2)^{-1/2+2\varepsilon} \|\psi\|_{r^\prime}^2
\]
for $t-s>2$.
We then apply Lemma \ref{growth-est} to find the growth estimate
\begin{equation}
\int_{(s+t)/2}^{t-1}\|\nabla T(t,\tau)^*\psi\|_2\,d\tau
=\int_1^{(t-s)/2}\|\nabla w(\tau)\|_2\,d\tau
\leq C_\varepsilon (t-s-2)^{1/4+\varepsilon}\|\psi\|_{r^\prime}
\label{J-second}
\end{equation}
for $t-s>2$, from which the second integral of \eqref{J-split} is estimated as
\begin{equation*}
\begin{split}
\int_{(s+t)/2}^{t-1}
&\leq C(t-s)^{-3/2r}\int_{(s+t)/2}^{t-1}\|\nabla T(t,\tau)^*\psi\|_2\,d\tau  \\
&\leq C_\varepsilon (t-s)^{-3/2r+1/4+\varepsilon}\|\psi\|_{r^\prime}
\end{split}
\end{equation*}
for $t-s>3$, while the first one of \eqref{J-split} is discussed
by use of \eqref{J-first} in the similar way to the previous step as
\begin{equation*}
\begin{split}
\int_{s+1}^{(s+t)/2}
&\leq C(t-s)^{(1-3/r)/2}
\left(\int_{s+1}^{(s+t)/2}\|\nabla T(t,\tau)^*\psi\|_2^2\,d\tau\right)^{1/2} \\
&\leq C_\varepsilon (t-s)^{-3/2r+1/4+\varepsilon}\|\psi\|_{r^\prime}
\end{split}
\end{equation*}
for $t-s>3$.
In view of \eqref{J} with \eqref{J-split} we deduce from both estimates above that
\[
|J|\leq C\|f\|_r\|\psi\|_{r^\prime}
\]
for $t-s>3$ with some constant $C=C(m,r,\theta,D)>0$ provided $3<r<6$.
This together with \eqref{near-end} as well as \eqref{top}
gives \eqref{unif-bdd} for such $r$ (there is no need to fill in
the case $r=3$ although one can do that by interpolation).
We thus obtain \eqref{LqLr-adj} with $6/5<q\leq r\leq 2$ for all $t>s\geq 0$
by Lemma \ref{auxi}.

Suppose $6\leq r<\infty$.
Then \eqref{weak-d1} can be improved as
\begin{equation}
\|T(t,\tau)^*\psi\|_2
\leq C_\varepsilon (t-s-1)^{-1/2+\varepsilon}\|\psi\|_{r^\prime}
\label{weak-d2}
\end{equation}
for $t-s>1$, where $\varepsilon >0$ is arbitrary.
With \eqref{weak-d2} in hand, estimates \eqref{J-first} and \eqref{J-second}
can be respectively replaced by
\[
\int_{s+1}^{(s+t)/2}\|\nabla T(t,\tau)^*\psi\|_2^2\,d\tau
\leq C_\varepsilon (t-s-2)^{-1+2\varepsilon}\|\psi\|_{r^\prime}^2
\]
and by
\[
\int_{(s+t)/2}^{t-1}\|\nabla T(t,\tau)^*\psi\|_2\,d\tau
\leq C_\varepsilon (t-s-2)^\varepsilon\|\psi\|_{r^\prime}
\]
for $t-s>2$, where the latter follows from the former owing to
Lemma \ref{growth-est}.
Then the same argument with use of splitting \eqref{J-split}
as in the second step leads to
\[
|J|\leq C_\varepsilon (t-s)^{-3/2r+\varepsilon}\|f\|_r\|\psi\|_{r^\prime}
\leq C\|f\|_r\|\psi\|_{r^\prime}
\]
for $t-s>3$ with some constant $C=C(m,r,\theta,D)>0$ when choosing
an appropriate $\varepsilon >0$ for given $r\in [6,\infty)$.
We collect \eqref{top}, \eqref{near-end} and the last estimate to
furnish \eqref{unif-bdd} for every $r\in [6,\infty)$.
Hence Lemma \ref{auxi} concludes
\eqref{LqLr} with $2\leq q\leq r<\infty$ as well as
\eqref{LqLr-adj} with $1<q\leq r\leq 2$ for all $t>s\geq 0$.
\begin{remark}
In 2D case the argument above does not work even in the first step.
In fact, the decay rate $(t-s)^{-3/2q}$ of $\|F(t)\|_q$
must be replaced by the slower one $(t-s)^{-1/q}$ in \eqref{force}
and, thereby, the last integral of \eqref{J-2} becomes
$\int_{s+1}^{t-1}(\tau -s)^{-2/r}d\tau$
which cannot be bounded because $r>2$.
Thus we do not know \eqref{unif-bdd} even if $r\in (2,\infty)$
is close to $2$.
The difficulty of 2D case is described in \cite{Hi16}
(even for the autonomous Oseen system) from
the viewpoint of spectral analysis.
\label{2dim}
\end{remark}

It remains to show the opposite case, that is,
\eqref{LqLr} with $1<q\leq r\leq 2$ and \eqref{LqLr-adj} with
$2\leq q\leq r<\infty$.
We fix $t>3$.
Given $g\in C^\infty_{0,\sigma}(D)$, we describe the solution
$T(t,s)^*g$ to the backward system \eqref{backward} in the form
\begin{equation}
T(t,s)^*g=(1-\phi)U(t,s)^*g
+\mathbb B\left[\big(U(t,s)^*g\big)\cdot\nabla\phi\right]+u(s),
\label{decompo-adj}
\end{equation}
and intend to estimate the perturbation $u(s)=u(s;t)$.
Here, $\phi$ is the same cut-off function as in \eqref{decompo} and
$\mathbb B$ denotes the same Bogovskii operator there.
Recall that $U(t,s)^*g$ is the solution to the whole space problem
\eqref{adj-wh}--\eqref{adj-wh-sub} with the final velocity $g$ at the final time $t$.
Then the function $u(s)$ obeys
\begin{equation*}
\begin{split}
-\partial_su
&=\Delta u-(\eta(s)+\omega(s)\times x)\cdot\nabla u+\omega(s)\times u+\nabla p+G, \\
\mbox{div $u$}&=0,
\end{split}
\end{equation*}
in $D\times [0,t]$ subject to
\begin{equation*}
\begin{split}
&u|_{\partial D}=0, \\
&u\to 0\quad\mbox{as $|x|\to\infty$},  \\
&u(\cdot,t)=\widetilde g:=\phi g-\mathbb B[g\cdot\nabla\phi],
\end{split}
\end{equation*}
where $p(s)$ stands for the pressure associated with $T(t,s)^*g$ and
\begin{equation*}
\begin{split}
G(x,s)
&=-2\nabla\phi\cdot\nabla U(t,s)^*g
-[\Delta\phi-(\eta(s)+\omega(s)\times x)\cdot\nabla\phi]\,U(t,s)^*g  \\
&\quad 
+\mathbb B\left[\partial_s\big(U(t,s)^*g\big)\cdot\nabla\phi\right]
+\Delta\mathbb B\left[\big(U(t,s)^*g\big)\cdot\nabla\phi\right]  \\
&\quad
-(\eta(s)+\omega(s)\times x)\cdot\nabla
\mathbb B\left[\big(U(t,s)^*g\big)\cdot\nabla\phi\right]  \\
&\quad
+\omega(s)\times\mathbb B\left[\big(U(t,s)^*g\big)\cdot\nabla\phi\right],
\end{split}
\end{equation*}
which satisfies
\begin{equation}
\|G(s)\|_q\leq
\left\{
\begin{array}{ll}
C(m+1)(t-s)^{-1/2}\|g\|_q,
&0<t-s<1, \\
C(m+1)(t-s)^{-3/2q}\|g\|_q, \qquad
&t-s\geq 1,
\end{array}
\right.
\label{force-adj}
\end{equation}
for $1<q<\infty$.
On account of \eqref{adj-reg}, \eqref{adj-reg-wh} and Lemma \ref{bogov} we have
\[
u\in C^1([0,t]; L^q_\sigma(D))
\]
as well as $u(s)\in Y_q(D)$ for every $q\in (1,\infty)$.
Let $\psi\in C^\infty_{0,\sigma}(D)$, then we see from \eqref{evo-1}
and \eqref{dual-1} that
\begin{equation*}
\begin{split}
\partial_\tau\langle\psi, T(\tau,s)^*u(\tau)\rangle_D
&=\partial_\tau\langle T(\tau,s)\psi, u(\tau)\rangle_D  \\
&=\langle T(\tau,s)\psi, \partial_\tau u(\tau)\rangle_D
+\langle -L_+(\tau)T(\tau,s)\psi, u(\tau)\rangle_D  \\
&=\langle T(\tau,s)\psi, \partial_\tau u(\tau)-L_-(\tau)u(\tau)\rangle_D  \\
&=-\langle T(\tau,s)\psi, G(\tau)\rangle_D,
\end{split}
\end{equation*}
which leads to the Duhamel formula in the weak form
\begin{equation}
\langle\psi, u(s)\rangle_D
=\langle T(t,s)\psi, \widetilde g\rangle_D
+\int_s^t\langle T(\tau,s)\psi, G(\tau)\rangle_D\,d\tau.
\label{duhamel-2}
\end{equation}

Let $r\in (2,\infty)$, then our task is to derive
\begin{equation}
\|u(s)\|_r\leq C\|g\|_r  \qquad
\mbox{for $t-s>3$},
\label{unif-bdd-2}
\end{equation}
as in \eqref{unif-bdd}.
The corresponding claim to Lemma \ref{auxi} is the following,
whose proof is essentially the same.
\begin{lemma}
In addition to the conditions in Theorem \ref{main},
suppose that, with some $r_0\in (2,\infty)$, estimate
\eqref{unif-bdd-2} holds for all $g\in C^\infty_{0,\sigma}(D)$.

\begin{enumerate}
\item
Let $r_0^\prime\leq q\leq r\leq 2$, where $1/r_0^\prime+1/r_0=1$.
Then there is a constant $C=C(m,q,r,r_0,\theta,D)>0$ such that
\eqref{LqLr} holds for all $t>s\geq 0$ and $f\in L^q_\sigma(D)$.

\item
Let $2\leq q\leq r\leq r_0$.
Then there is a constant $C=C(m,q,r,r_0,\theta,D)>0$ such that
\eqref{LqLr-adj} holds for all $t>s\geq 0$ and $g\in L^q_\sigma(D)$.

\end{enumerate}
\label{auxi-2}
\end{lemma}

As in \eqref{top} and \eqref{near-end},
we use \eqref{energy}, \eqref{force-adj} and 
Proposition \ref{unif-loc} to obtain
\[
|\langle T(t,s)\psi, \widetilde g\rangle_D|
+\left|
\left(\int_s^{s+1}+\int_{t-1}^t\right)\langle T(\tau,s)\psi, G(\tau)\rangle_D\,d\tau
\right|
\leq C\|g\|_r\|\psi\|_{r^\prime}
\]
for $t-s>3$ and every $r\in (2,\infty)$, where $1/r^\prime+1/r=1$.
If $2<r<3$, then we observe
\[
\left|\int_{s+1}^{t-1}\langle T(\tau,s)\psi, G(\tau)\rangle_D\,d\tau\right|
\leq C\|g\|_r\|\psi\|_{r^\prime}
\]
for $t-s>3$ by the same way as in the treatment of \eqref{J}.
In view of \eqref{duhamel-2} we obtain \eqref{unif-bdd-2} for $r\in (2,3)$,
which gives \eqref{LqLr} with $3/2<q\leq r\leq 2$ for all $t>s\geq 0$
by Lemma \ref{auxi-2}.

We next proceed to the case $r\in (3,6)$ by splitting the integral over
$(s+1,t-1)$ as in \eqref{J-split},
where estimates \eqref{J-first} and \eqref{J-second} are replaced by
\[
\int_{s+1}^{(s+t)/2}\|\nabla T(\tau,s)\psi\|_2\,d\tau
\leq C_\varepsilon (t-s-2)^{1/4+\varepsilon}\|\psi\|_{r^\prime}
\]
and
\begin{equation*}
\begin{split}
\int_{(s+t)/2}^{t-1}\|\nabla T(\tau,s)\psi\|_2^2\,d\tau
&\leq\frac{1}{2}\|T((s+t)/2,s)\psi\|_2^2  \\
&\leq C_\varepsilon (t-s-2)^{-1/2+2\varepsilon}\|\psi\|_{r^\prime}^2
\end{split}
\end{equation*}
for $t-s>2$, where $\varepsilon >0$ is arbitrarily small;
in fact, the latter implies the former by Lemma \ref{growth-est}.
Then the same argument as before yields \eqref{unif-bdd-2}
for $r\in (3,6)$ and, therefore, \eqref{LqLr} with
$6/5<q\leq r\leq 2$.

Repeating this argument once more by use of \eqref{LqLr} with such $(q,r)$,
we find \eqref{LqLr-adj} with $2\leq q\leq r<\infty$ as well as
\eqref{LqLr} with $1<q\leq r\leq 2$
for all $t>s\geq 0$.
Finally, the remaining case $q<2<r$ for both estimates
is obvious on account of semigroup properties \eqref{semi}, \eqref{b-semi}.
The proof of Theorem \ref{main} is complete.
\hfill
$\Box$

\section{Application to the Navier-Stokes problem}
\label{appli}

{\em 5.1. How to construct the Navier-Stokes flow}

Let us apply the decay estimates of the evolution operator obtained in
Theorems \ref{main} and \ref{int-decay} to the Navier-Stokes 
initial value problem \eqref{NS}--\eqref{IC}.
Concerning the behavior of
the translational and angular velocities $\eta,\, \omega$,
we could consider several situations,
for instance, the one in which they converge to some constant vectors
as $t\to\infty$, the one in which they oscillate, and so on.
The only claim we are going to show is the
stability of the rest state $u=0$ in the simplest situation
in which the motion of the rigid body becomes slow
as time goes on, that is, its velocity
$\eta +\omega\times x$ tends to zero as $t\to\infty$.
The stability of nontrivial states is of course more involved
as well as interesting and will be discussed elsewhere.
The emphasis is how to treat the nonlinearity
$u\cdot\nabla u$ (toward analysis of stability of such states)
rather than the result (stability of the rest state) itself.
This is by no means obvious because of lack of pointwise decay estimate
\eqref{LqLr-grad} for the gradient of the evolution operator
as $(t-s)\to\infty$.
In order to make the idea clearer, first of all, 
it would be better to consider the problem
\eqref{NS}--\eqref{IC} in which the no-slip boundary condition
\eqref{noslip} is replaced by the homogeneous one
$u|_{\partial D}=0$ although this modification only on $\partial D$
is not physically relevant.
We will then discuss the right problem \eqref{NS}--\eqref{IC}
in the next subsection.

The problem \eqref{NS}, \eqref{sp-infty}, \eqref{IC} subject to
$u|_{\partial D}=0$ is formulated as the initial value problem
\[
\partial_tu+L_+(t)u+P(u\cdot\nabla u)=0, \qquad t\in(0,\infty); \qquad
u(0)=u_0,
\]
in $L^q_\sigma(D)$, which is formally converted to the integral equation
\begin{equation}
u(t)=T(t,0)u_0-\int_0^t T(t,\tau)P(u\cdot\nabla u)(\tau)\,d\tau
\label{INS}
\end{equation}
and even to its weak form
(see \cite{KoO}, \cite{BM95}, \cite{KoY}, \cite{Y})
\begin{equation}
\begin{split}
\langle u(t), \psi\rangle_D
=\langle T(t,0)u_0, \psi\rangle_D
+\int_0^t\langle (u\otimes u)(\tau),
&\nabla T(t,\tau)^*\psi\rangle_D \,d\tau, \\
&\forall\,\psi\in C^\infty_{0,\sigma}(D),
\end{split}
\label{w-INS}
\end{equation}
by using the divergence structure
$u\cdot\nabla u=\mbox{div $(u\otimes u)$}$.

We intend to solve \eqref{w-INS} globally in time for
$u_0\in L^3_\sigma(D)$ with small $\|u_0\|_3$.
Let $r\in (3,\infty)$.
Under the same conditions on $(\eta,\omega)$ as in Theorem \ref{main},
there is a constant $c_r=c_r(m,\theta,D)>0$ such that
\begin{equation}
\|T(t,0)u_0\|_r
\leq c_rt^{-1/2+3/2r}\|u_0\|_3
\label{top-NS}
\end{equation}
for all $t>0$, in view of which it is reasonable to seek a
solution of class
\begin{equation}
\begin{split}
E_r:=\big\{
u\in C_w((0,\infty); L^r_\sigma(D));\;\;
t^{1/2-3/2r}u\in 
&L^\infty(0,\infty; L^r_\sigma(D)),  \\
&\lim_{t\to 0}\|u\|_{E_r(t)}=0 \big\},
\end{split}
\label{where}
\end{equation}
which is a Banach space endowed with norm
\[
\|\cdot\|_{E_r}:=\sup_{t>0}\|\cdot\|_{E_r(t)},
\]
where
\[
\|u\|_{E_r(t)}=\sup_{0<\tau\leq t}\tau^{1/2-3/2r}\|u(\tau)\|_r, \qquad t>0.
\]
Especially, the case $r=4$ plays a role because we know \eqref{grad-decay}.
In fact, given $u,\, v\in E_4$, we have
\begin{equation}
\begin{split}
&\quad \left|\int_0^t\langle (u\otimes v)(\tau), 
\nabla T(t,\tau)^*\psi\rangle_D\,d\tau\right|  \\
&\leq \|u\|_{E_4(t)}\|v\|_{E_4(t)}\int_0^t \tau^{-1/4}
\|\nabla T(t,\tau)^*\psi\|_2\,d\tau
\end{split}
\label{duha-NS}
\end{equation}
for all $t>0$.
The key observation is the following.
\begin{lemma}
Let $q\in (6/5,2]$.
Under the same conditions as in Theorem \ref{main}, there is a constant
$C=C(m,q,\theta,D)>0$ such that
\begin{equation}
\int_0^t \tau^{-1/4}\|\nabla T(t,\tau)^*\psi\|_2\,d\tau
\leq Ct^{-3/2q+1}\|\psi\|_q
\label{key-0}
\end{equation}
for all $t>0$ and $\psi\in L^q_\sigma(D)$
whenever \eqref{mag} is satisfied.
\label{key}
\end{lemma}

\begin{proof}
We fix $\psi\in L^q_\sigma(D)$ as well as $t>0$ and set
$w(t-\tau)=T(t,\tau)^*\psi$ by following the notation \eqref{adj-2}.
We split the integral into
\begin{equation*}
\begin{split}
&\quad \int_0^t \tau^{-1/4}\|\nabla T(t,\tau)^*\psi\|_2\,d\tau \\
&=\left(\int_0^{t/2}+\int_{t/2}^t\right)(t-\tau)^{-1/4}\|\nabla w(\tau)\|_2\,d\tau
=:I_1+I_2.
\end{split}
\end{equation*}
By \eqref{grad-decay} we have
\begin{equation}
\int_{t/2}^t\|\nabla w(\tau)\|_2^2\,d\tau
=\int_0^{t/2}\|\nabla T(t,\tau)^*\psi\|_2^2\,d\tau
\leq Ct^{-3/q+3/2}\|\psi\|_q^2
\label{key-1}
\end{equation}
for $1<q\leq 2$.
As in the proof of \eqref{J-second}, we apply Lemma \ref{growth-est}
to obtain the growth estimate
\begin{equation}
\int_0^{t/2}\|\nabla w(\tau)\|_2\,d\tau
\leq Ct^{5/4-3/2q}\|\psi\|_q
\label{key-2}
\end{equation}
as long as $6/5<q\leq 2$ so that $3/q-3/2<1$.
It thus follows from \eqref{key-1} and \eqref{key-2} that
\[
I_1\leq Ct^{-1/4}\int_0^{t/2}\|\nabla w(\tau)\|_2\,d\tau
\leq Ct^{-3/2q+1}\|\psi\|_q
\]
and that
\[
I_2\leq Ct^{1/4}\left(\int_{t/2}^t\|\nabla w(\tau)\|_2^2\,d\tau\right)^{1/2}
\leq Ct^{-3/2q+1}\|\psi\|_q.
\]
The proof is complete.
\end{proof}

Given $u\in E_4$ and $t>0$, let us define $(Hu)(t)$ by
\[
\langle (Hu)(t), \psi\rangle_D=
\mbox{the RHS of \eqref{w-INS}}, \qquad
\forall\,\psi\in C^\infty_{0,\sigma}(D),
\]
and regard \eqref{w-INS} as the equation $u=Hu$.
By \eqref{duha-NS} with \eqref{key-0} and by \eqref{top-NS}
we find from \eqref{short} that
\begin{equation}
\begin{split}
t^{1/2-3/2r}\|(Hu)(t)\|_r
&\leq t^{1/2-3/2r}\|T(t,0)u_0\|_r+k_r\|u\|_{E_4(t)}^2 \to 0 \quad (t\to 0), \\
\|(Hu)(t)-u_0\|_3
&\leq \|T(t,0)u_0-u_0\|_3+k_3\|u\|_{E_4(t)}^2 \to 0 \quad (t\to 0),
\end{split}
\label{near-zero}
\end{equation}
for $r\in (3,6)$ and that
\begin{equation}
\sup_{t>0}\,t^{1/2-3/2r}\|(Hu)(t)\|_r \leq c_r\|u_0\|_3+k_r\|u\|_{E_4}^2,
\label{est-NS}
\end{equation}
for $r\in [3,6)$ with some constant $k_r=k_r(m,\theta,D)>0$.

Let us show the weak-continuity of $Hu$ with respect to $t\in (0,\infty)$
with values in $L^4_\sigma(D)$.
We fix ${\cal T}\in (0,\infty)$ arbitrarily.
Let $t\in (0,{\cal T})$ and $t+h\in (t/2,{\cal T})$.
On account of $T(t/2,0)u_0\in L^4_\sigma(D)$,
it is obvious that
\[
\|T(t+h,0)u_0-T(t,0)u_0\|_4
=\|\{T(t+h,t/2)-T(t,t/2)\}T(t/2,0)u_0\|_4\to 0
\]
as $h\to 0$.
When $t<t+h<{\cal T}$, the second part of
$\langle(Hu)(t+h)-(Hu)(t), \psi\rangle$
is splitted into
\begin{equation}
\begin{split}
I+J:=
&\int_0^t\langle (u\otimes u)(\tau),
\nabla\{T(t+h,\tau)^*-T(t,\tau)^*\}\psi\rangle\,d\tau  \\
&+\int_t^{t+h}\langle (u\otimes u)(\tau),
\nabla T(t+h,\tau)^*\psi\rangle\,d\tau.
\end{split}
\label{diff-2nd}
\end{equation}
Since we know the $L^q$-$L^r$ estimate of $\nabla T(t,\tau)^*$
for $0\leq\tau <t<{\cal T}$,
we find that
\[
|I|\leq C\|u\|_{E_4}^2\, t^{-1/8}\|T(t+h,t)^*\psi-\psi\|_{4/3}
\]
and that
\[
|J|\leq C\|u\|_{E_4}^2\, t^{-1/4}h^{1/8}\|\psi\|_{4/3}
\]
with some constant $C=C({\cal T})>0$.
The other case $t/2<t+h<t<{\cal T}$, 
in which \eqref{diff-2nd} should be replaced by
\begin{equation*}
\begin{split}
I+J:=
&\int_0^{t+h}\langle (u\otimes u)(\tau),
\nabla\{T(t+h,\tau)^*-T(t,\tau)^*\}\psi\rangle\,d\tau   \\
&-\int_{t+h}^t\langle (u\otimes u)(\tau),
\nabla T(t,\tau)^*\psi\rangle\,d\tau,
\end{split}
\end{equation*}
is discussed similarly to obtain
\[
|I|\leq C\|u\|_{E_4}^2 (t/2)^{-1/8}\|T(t,t+h)^*\psi-\psi\|_{4/3}
\]
and
\[
|J|\leq C\|u\|_{E_4}^2 (t/2)^{-1/4}(-h)^{1/8}\|\psi\|_{4/3}.
\]
We are thus led to
$Hu\in C_w((0,\infty); L^4_\sigma(D))$, namely,
it is weakly continuous with values in $L^4_\sigma(D)$
(although the estimates above for $I$ tell us that the strong-continuity
would not be clear, where the difficulty stems from the fact that the
corresponding autonomous operator is not a generator of analytic
semigroups unless $\omega =0$).

As a consequence, we obtain $Hu\in E_4$ with
\[
\|Hu\|_{E_4}\leq c_4\|u_0\|_3+k_4\|u\|_{E_4}^2.
\]
Moreover, for $u,\, v\in E_4$, we have
\[
\|Hu-Hv\|_{E_4}\leq k_4\left(\|u\|_{E_4}+\|v\|_{E_4}\right)\|u-v\|_{E_4}.
\]
In this way, we get a unique solution $u\in E_4$ to \eqref{w-INS}
with
\[
\|u\|_{E_4}\leq\frac{1-\sqrt{1-4c_4k_4\|u_0\|_3}}{2k_4}
<2c_4\|u_0\|_3
\]
provided that
$\|u_0\|_3<1/(4c_4k_4)$.
The initial condition
$\lim_{t\to 0}\|u(t)-u_0\|_3=0$ follows from \eqref{near-zero}.
Besides the $L^4$ decay, we obtain the $L^r$ decay of the solution
with rate $t^{-1/2+3/2r}$ on account of \eqref{est-NS} as long as $3<r<6$.
\bigskip

\noindent
{\em 5.2. A global existence theorem}

Finally, we will provide a global existence theorem for the initial
value problem \eqref{NS}--\eqref{IC}.
Let $\phi$ be the same cut-off function as taken at the beginning of
Section \ref{proof} and set
\[
b(x,t)=\frac{1}{2}\,
\mbox{rot $\left\{\phi(x)\big(\eta(t)\times x-|x|^2\omega(t)\big)\right\}$},
\]
which fulfills
\[
\mbox{div $b$}=0, \qquad
b|_{\partial D}=\eta+\omega\times x, \qquad
b(t)\in C_0^\infty(B_{3R_0}),
\]
where $R_0$ is fixed as in \eqref{cov}.
Let us look for a solution of the form
\begin{equation}
u(x,t)=b(x,t)+v(x,t)
\label{subt}
\end{equation}
to \eqref{NS}--\eqref{IC}.
Then, instead of \eqref{w-INS}, $v(t)$ should obey
\begin{equation}
\begin{split}
\langle v(t), \psi\rangle_D
&=\langle T(t,0)v_0, \psi\rangle_D
+\int_0^t\langle T(t,\tau)F(\tau), \psi\rangle_D\, d\tau  \\
&\quad +\int_0^t\langle (v\otimes v+v\otimes b+b\otimes v)(\tau),
\nabla T(t,\tau)^*\psi\rangle_D\,d\tau,  \\
&\qquad\qquad\qquad\qquad\qquad
\forall\,\psi\in C^\infty_{0,\sigma}(D),
\end{split}
\label{w-INS2}
\end{equation}
where
\[
v_0:=u_0-b(\cdot,0)\in L^3_\sigma(D),
\]
see \eqref{i-data} below, and
\[
F:=\Delta b+(\eta+\omega\times x)\cdot\nabla b-\omega\times b
-\partial_tb-b\cdot\nabla b.
\]
\begin{theorem}
Suppose that there is a constant $\gamma\in [1/8,1)$ satisfying
\begin{equation}
\begin{split}
&\eta,\,\omega\in C^1([0,\infty); \mathbb R^3),  \\
&M:=\sup_{t\geq 0}\; (1+t)^\gamma
\big(|\eta(t)|+|\eta^\prime(t)|+|\omega(t)|+|\omega^\prime(t)|\big)<\infty,
\label{body-1}
\end{split}
\end{equation}
and that
\begin{equation}
u_0\in L^3(D), \quad
\mbox{div $u_0$}=0, \quad
\nu\cdot (u_0-\eta(0)-\omega(0)\times x)|_{\partial D}=0.
\label{i-data}
\end{equation}
Then there is a constant $\delta=\delta(D)>0$ such that if
\[
\|u_0\|_3+M\leq\delta,
\]
problem \eqref{w-INS2} admits a unique global solution $v\in E_4$ which enjoys
\begin{equation}
\|v(t)\|_r=O(t^{-\mu}) \quad
\mbox{as $t\to\infty$}
\label{v-decay}
\end{equation}
for every $r\in [3,6)$ with
$\mu:=\min\{1/2-3/2r,\, \gamma\}$,
where $E_4$ is given by \eqref{where}.
\label{main-NS}
\end{theorem}

\begin{proof}
As in the previous subsection the only point is to employ Lemma \ref{key}
and so the proof may be almost obvious for most of readers,
nevertheless it will be presented in order to clarify
why $\gamma\geq 1/8$.
We may assume $M\leq 1$ at the beginning, then we have \eqref{mag} with $m=3$
(and $\theta=1$); in what follows, we use Theorems \ref{main} and \ref{int-decay}
for such $m$.
By \eqref{body-1} we obtain
\begin{equation}
\|b(t)\|_q+\|F(t)\|_q\leq CM(1+t)^{-\gamma}
\label{bF}
\end{equation}
for all $t\geq 0$ and $q\in (1,\infty]$ with some constant $C=C(q)>0$,
which implies that
\begin{equation}
\int_0^t\|T(t,\tau)F(\tau)\|_r\,d\tau
\leq \alpha_r(t):=
\left\{
\begin{array}{ll}
CMt(1+t)^{-1-\gamma}, & r\in (3,\infty), \\
C_\varepsilon Mt(1+t)^{-1-\gamma+\varepsilon}, \qquad & r=3,
\end{array}
\right.
\label{gomi}
\end{equation}
for all $t>0$
with some constants $C=C(r)>0$ and $C_\varepsilon >0$,
where $\varepsilon >0$ is arbitrary.
This decay estimate for $t>2$ is easily 
verified by splitting the integral into three parts
\[
\int_0^t=\int_0^{t/2}+\int_{t/2}^{t-1}+\int_{t-1}^t
\]
and by using \eqref{bF} with $q\in (1,r)$ satisfying
$(3/q-3/r)/2>1$ (which is possible as long as $r>3$)
except for the last integral over $(t-1,t)$.
Given $v\in E_4$ and $t>0$, this time, we define $({\cal H}v)(t)$ by
\[
\langle ({\cal H}v)(t), \psi\rangle_D
=\mbox{the RHS of \eqref{w-INS2}}, \qquad
\forall\,\psi\in C^\infty_{0,\sigma}(D).
\]
Then we have $({\cal H}v)(t)\in L^r_\sigma(D)$ for all $r\in [3,6)$;
further, \eqref{near-zero} and \eqref{est-NS} are respectively replaced by
\begin{equation*}
\begin{split}
t^{1/2-3/2r}\|({\cal H}v)(t)\|_r
&\leq t^{1/2-3/2r}\|T(t,0)v_0\|_r +\beta_r(t)
\to 0 \quad (t\to 0),  \\
\|({\cal H}v)(t)-v_0\|_3
&\leq \|T(t,0)v_0-v_0\|_3 +\beta_3(t)
\to 0 \quad (t\to 0),
\end{split}
\end{equation*}
(the former holds for $r\in (3,6)$) with
\[
\beta_r(t):=
t^{1/2-3/2r}\alpha_r(t)
+k_r\left(CM+\|v\|_{E_4(t)}\right)\|v\|_{E_4(t)}
\]
and by
\begin{equation*}
\begin{split}
\|({\cal H}v)(t)\|_r
&\leq c_r(\|u_0\|_3+CM)\,t^{-1/2+3/2r}+\alpha_r(t) \\
&\quad +k_r\left(CM+\|v\|_{E_4}\right)\|v\|_{E_4}\,t^{-1/2+3/2r}
\end{split}
\end{equation*}
for all $t>0$, where $\alpha_r(t)$ is given by \eqref{gomi}.
One can also verify
${\cal H}v\in C_w((0,\infty); L^4_\sigma(D))$
along the same way as in the previous subsection, so that
${\cal H}v\in E_4$.
Consequently, we see that
\begin{equation*}
\begin{split}
\|{\cal H}v\|_{E_4}
&\leq C(\|u_0\|_3+M)+k_4\left(CM+\|v\|_{E_4}\right)\|v\|_{E_4}, \\
\|{\cal H}v-{\cal H}w\|_{E_4}
&\leq k_4\left(CM+\|v\|_{E_4}+\|w\|_{E_4}\right)\|v-w\|_{E_4},
\end{split}
\end{equation*}
for $v,\, w\in E_4$, which completes the proof.
\end{proof}

In view of \eqref{subt}, \eqref{v-decay} and \eqref{bF}, we conclude
\[
\|u(t)\|_r=O(t^{-\mu}) \quad\mbox{as $t\to\infty$}
\]
for every $r\in [3,6)$ with the same $\mu$ as in \eqref{v-decay}.
\bigskip

\noindent
{\em Acknowledgments}.
The author is partially supported by Grant-in-Aid for scientific research
15K04954 from the Japan Society for the Promotion of Science.


\end{document}